\documentclass[twoside]{article}

\usepackage{amsthm, amsmath, amssymb, mathtools}
\let\saverem=\rem
\let\rem\relax
\usepackage[squaren]{SIunits}
\let\rem=\saverem
\usepackage{ucs}
\usepackage[utf8,utf8x]{inputenc}
\usepackage[T1]{fontenc}
\usepackage[english]{babel}

\usepackage[twoside,
  paperwidth=210mm,
  paperheight=297mm,
  textheight=682pt,
  textwidth=500pt,
  centering,
  headheight=50pt,
  headsep=12pt,
 ]
 {geometry}

\usepackage{RR}
\RRauthor{
 Olivier Huber
 \and Vincent Acary
 \and Bernard Brogliato
}

\RRtitle{Analyse de la discretization explicite ou implicite des controlleurs par modes glissants}
\RRprojet{BiPoP}
\RCGrenoble
\RRdate{October 2013}
\RRmotcle{contrôle par modes glissants,
système échantionnés, stabilité de Lyapunov en temps discret, 
discretisation implicite}
\RRkeyword{sliding mode control, sampled-data system, 
discrete-time Lyapunov stability,implicit discretization}   

\usepackage{gitinfo}
\usepackage{grffile}
\usepackage[skip=0pt]{subcaption}
\usepackage{floatrow}
\usepackage{cases}
\usepackage{pgf}

\usepackage{footnote}
\usepackage{cite}

\usepackage{hyperref}
\hypersetup{
pdfauthor = {Olivier Huber},
colorlinks,%
citecolor=black,%
filecolor=black,%
linkcolor=black,%
urlcolor=black
}

\makeatletter
\let\cl@part\relax
\makeatother

\usepackage[noabbrev]{cleveref}
\newcommand{\crefrangeconjunction}{--}
\newcommand{\crefpairconjunction}{--}
\newcommand{\crefmiddleconjunction}{ and~}
\newcommand{\creflastconjunction}{--}
\makeatletter
\@ifclassloaded{memoir}
{
\nonzeroparskip
\setsecnumdepth{subsection}
}{}
\makeatother

\providecommand{\N}{\mathbb{N}}
\providecommand{\R}{\mathbb{R}}

\makeatletter
\@ifclassloaded{memoir}
{
\theoremstyle{definition}
\newtheorem{de}{Definition}[chapter]
\newtheorem{defr}{Définition}[chapter]

\theoremstyle{plain}
\newtheorem{thmfr}{Théorème}[chapter]

\newtheorem{ppty}{Property}[chapter]
\newtheorem{prop}{Proposition}[chapter]
\newtheorem{dem}{Démonstration}[chapter]
\newtheorem{coro}{Corollaire}[chapter]
\newtheorem{lem}{Lemma}[chapter]
}
{
\theoremstyle{definition}
\newtheorem{de}{Definition}

\theoremstyle{plain}

\newtheorem{prop}{Proposition}

\newtheorem{lem}{Lemma}
}

\makeatother

\theoremstyle{remark}

\newtheorem{remark}{Remark}

\DeclareMathOperator{\sgn}{sgn}
\DeclareMathOperator{\Sgn}{Sgn}
\DeclareMathOperator{\sat}{sat}

\DeclareMathOperator{\esssup}{ess\,sup}
\DeclareMathOperator{\proj}{proj}

\newcommand{\sate}{\sat_{\varepsilon}}

\def\lenSubFig{.49}

\pgfrealjobname{RR-8383}

\makeatletter
\providecommand*{\toclevel@algorithm}{0}
\makeatother
\allowdisplaybreaks

\reversemarginpar
\graphicspath{ {img/} }
\RRetitle{
 Analysis of explicit and implicit discrete-time equivalent-control based sliding mode controllers 
}

\abstract{%
Different time-discretization methods for equivalent-control based sliding mode control (ECB-SMC) are presented.
A new discrete-time sliding mode control scheme is proposed for linear time-invariant (LTI) systems.
It is error-free in the discretization of the equivalent part of the control input.
Results from simulations using the various discretized SMC schemes are shown, with and without perturbations.
They illustrate the different behaviours that can be observed.
Stability results for the proposed scheme are derived.}

\RRresume{
 Dans ce rapport, differentes techniques de discretisation temporelle pour le contrôle
 par modes glissants sont présentées. Une nouvelle structure de contrôle par modes glissants est
 proposée pour des systèmes linéaires invariants. Elle ne comporte pas d'erreur dans la discretisation de
 la partie équivalente du contrôle. Des résultats de simulations avec divers contrôlleurs par modes glissants
 sont présentées, avec ou sans perturbation. Ils illustrent les phénomènes qui peuvent être observés suivant
 le type de discretisation utilisé. Pour finir, la stabilité de la nouvelle structure de contrôle est étudiée.
}

\begin{document}
\RRNo{8383}
\makeRR

\section{Introduction}

The time discretization of sliding-mode controllers has witnessed an intense activity in the past 30 years
\cite{sarpturk1987stability, drakunov1989discrete, furuta1990sliding, utkin1994sliding, gao1995discrete, golo2000robust}
and \cite{milosavljevic1985general}.
This concerns in particular the classical Equivalent-Control-Based Sliding-Mode Control (ECB-SMC),
which consists of two sub-controllers: the state-continuous equivalent control $u^{eq}$ and the state-discontinuous control $u^{s}$.
In these past research efforts, most of the focus was on the discontinuous part of the control, since it introduces numerical chattering.
Several solutions to alleviate numerical chattering (that is solely due to the time discretization
\cite{galias2006complex,galias2008analysis,wang2009zoh,acary2010implicit,acary2012chattering})
have been proposed~\cite{sarpturk1987stability, drakunov1989discrete, furuta1990sliding, utkin1994sliding, gao1995discrete, golo2000robust,plestan2012advances,defoort2009novel,milosavljevic1985general},
 most of them consisting in the definition of a so-called quasi-sliding surface \cite{gao1995discrete}
 and an explicit discretization of $u^{s}$. The works in \cite{drakunov1989discrete} and \cite{golo2000robust}
depart from these discrete-time controllers and propose an algorithm which allows the sliding variable to take exactly
the zero value at sampling times.
They are however limited to first order, scalar systems and require some stringent assumptions.
Recently a new approach, which may be seen as a (non-trivial) extension of the controllers in
\cite{drakunov1989discrete} and \cite{golo2000robust}, has been proposed
in \cite{acary2010implicit} and \cite{acary2012chattering}.
The basic idea is to implement the discontinuous input $u^{s}$ in an implicit form,
while keeping its causality (i.e. the controller is nonanticipative).
Then this input has to be computed at each sampling time as the solution to a
generalized, set-valued equation, which takes the form of a simple projection on an interval in the simplest cases.
Let us illustrate the difference between explicit and implicit discretization with an academic example,
$\dot{x}(t)\in-\alpha\sgn(x(t))$, $\alpha>0$. We use the differential inclusion framework since
we let $\sgn(0)$ to take any value in $[-1,1]$ (this is formally stated in Definition~\ref{de:Sgn} as $\Sgn$).
An explicit discretization yields $x(t_{k+1}) \in x(t_k) - h\alpha\sgn(x(t_k))$,
whereas the implicit one yields $x(t_{k+1}) \in x(t_k) - h\alpha\sgn(x(t_{k+1}))$. As long as $|x(t_k)|\gg h\alpha$, there is
no difference between the two discretizations. But if $|x(t_k)| < \alpha h$, then the behaviour changes with the type
of discretization. With $0<x(t_k)<h\alpha$, in the explicit case, $x(t_{k+1}) \in x(t_k) - h\alpha\sgn(x(t_k)) < 0$.
The sign of the state will change at every $t_k$, leading to the well-known chattering phenomenon.
Whereas in the implicit case, it is possible to ensure $x(t_{k+1})=0$
by choosing $\sgn(x(t_{k+1})) = x(t_k)/(\alpha h)<1$.
The implicit discretization of the sign function is rigorously presented in Section~\ref{sec:dt}.

To the best of our knowledge, the discretization of the equivalent part has received little attention.
In this work, we present a study of the effects of discretization on both the equivalent and discontinuous
part of the control. After studying the different discretization methods and their
shortcomings, we propose a new discrete-time control scheme, where the equivalent part
is not a discretized version of its continuous-time counterpart
but is directly designed from the discrete-time dynamics.
Some properties of this scheme, like finite-time convergence to the sliding surface and perturbation attenuation,
are studied in Section~\ref{sec:stab_analysis}.

In this paper, we consider systems of the form
\begin{equation}
 \begin{cases}
\dot{x}(t) = Ax(t) + Bu(t) + B\xi(t)\\
u(t) = u^{eq}(t) + u^s(t)\\
\sigma(t) \coloneqq Cx(t)\\
u^s(t) \in -\alpha\Sgn\left(\sigma(x(t))\right),
\end{cases}
 \label{linSyst}
\end{equation}
with $x(t)\in\R^n, u(t)\in\R^p, \sigma(t)\in\R^p$, $C\in\R^{p\times n}$, and $\alpha>0$.
The function $\sigma$ is called the \emph{sliding variable}, the disturbance is denoted as $\xi$,
and $\Sgn$ is formally introduced in Definition~\ref{de:Sgn}.
The perturbation $\xi$ is supposed to be at least continuous: noise is not considered in this paper.
When $\xi=0$, the system is said to be \emph{nominal}.
The method used to discretize the dynamics is called Zero-Order Hold (ZOH),
also known as exact sampled-data representation.
It is often considered for technological reasons, but also because there is no error with this discretization
method.

In the remainder of this section, we introduce the notation. In Section~\ref{sec:recall} we briefly recall
the ECB-SMC theory. Then some classical discretization methods are presented in Section~\ref{sec:dt}.
Section~\ref{sec:perf} is dedicated to the discrete error analysis of various controllers of Section~\ref{sec:dt}.
We introduce our new discrete-time SMC scheme in Section~\ref{sec:design}.
Simulation results using different time-discretization methods are shown in Section~\ref{sec:sim},
to illustrate the possible different behaviours of the closed-loop system.
Finally, stability results are derived in Section~\ref{sec:stab_analysis}.
Conclusions end the paper in Section~\ref{sec:conclusions}.

\textbf{Notations}:
Let $\mathbf{x}\colon\R_+\times\R^p\times\R^n\to\R^n$ be the solution of system~\eqref{linSyst}, 
$x \coloneqq \mathbf{x}(\cdot, u, x_0)$ is the solution associated with a continuous-time control $u$ and an initial
state $x_0\in\R^n$, while $\bar{x} \coloneqq \mathbf{x}(\cdot, \bar{u}, x_0)$ is the solution with a step function $\bar{u}$
and the same initial state.
In the latter case, we denote by $\bar{\sigma} \coloneqq C\bar{x}$ the sliding variable.
The control values change at predefined time instants $t_k$,
defined for all $k\in\N:\; t_k \coloneqq t_0 + kh, \;t_0, h\in\R_+$. The scalar $h$ is called the timestep.
We denote $\bar{x}_k \coloneqq \bar{x}(t_k)$ and $\bar{\sigma}_k \coloneqq \bar{\sigma}(t_k)$ for all $k\in\N$.
For all $y\in\R^r$, $\|y\|_\infty = \max_i |y_i|$. For all $M\in\R^{r\times s}$, $\|M\|_\infty=\max_i \sum_j|M_{ij}|$.
Let $w\colon\R\to\R^r$ and $S$ be any interval in $\R$, $\|w\|_{\infty,S} = \max_i\esssup_{t\in S} |w_i(t)|$.
Let $\langle\cdot,\cdot\rangle$ denote the standard inner product in a Euclidean space and $\|\cdot\|$ the norm based upon it.
Let $\sgn$ be the classical single-valued sign function: for all $x>0,\; \sgn(x)=1, \; \sgn(-x) = -1$ and $\sgn(0) = 0$.
\begin{de}[Multivalued sign function]\label{de:Sgn}
Let $x\in\R$. The multivalued sign function $\Sgn\colon\R\rightrightarrows\R$ is defined as:
\begin{equation}
 \Sgn(x) = \begin{cases}
  1 & x>0\\
  -1& x<0\\
  [-1,1]&x=0.
 \end{cases}\hfill\label{eq:def_Sgn}
\end{equation}
If $x\in\R^n$, then the multivalued sign function $\Sgn\colon\R^n\rightrightarrows\R^n$ is defined as:
for all $j = 1,\dots,n, (\Sgn(x))_j \coloneqq \Sgn(x_j)$.
\end{de}
\begin{de}
 Let $f\colon \R^n\times\R\to\R^p$ and $l\in\R$. One has $f = \mathcal{O}(h^l)$ if for all $x\in\R^n$, there
 exists $c\in\R^p$ such that $f(x, h)/h^l\to c$ as $h\to0$.
\end{de}
\begin{de}{\cite[p.~147]{cottle2009linear}}\label{de:Pmatrix}
 Let $M\in\R^{n\times n}$. $M$ is a $\mathbf{P}$-matrix if for all $x\in \R^n$ such that for all $i\in\{1,\dots,n\}$,
 $x_i(Mx)_i \leq 0$, then $x = 0$.
\end{de}
\begin{lem}{\cite[p.~147]{cottle2009linear}}
 Let $M\in\R^{n\times n}$. If $M$ is positive-definite, then $M$ is a $\mathbf{P}$-matrix.
\end{lem}

\section{The equivalent-based continuous-time sliding-mode controller}\label{sec:recall}

Let us assume that the triplet $(A,B,C)$ has a strict vector relative degree $(1,1,\dots,1)$.
This implies that the decoupling matrix $CB$ is full rank.
The dynamics of the sliding variable in the nominal system~\eqref{linSyst} (that is with $\xi(t) = 0$) is
\begin{equation}
 \dot{\sigma}(t) = CAx(t) + CBu^{eq}(t) + CBu^s(t).
\end{equation}
The control law $u^{eq}$ is designed such that the system stays on the sliding surface once it has been reached
(in other word $u^{eq}$ renders the sliding surface invariant with $u^s\equiv0$):
\begin{equation}
 \dot{\sigma}(t) = 0 \text{ and } u^s(t) = 0 \quad \Rightarrow \quad u^{eq}(t) = -(CB)^{-1}CAx(t).
 \label{defUeq}
\end{equation}
Then the sliding variable dynamics with the equivalent control reduces to
\begin{equation}
 \left\{
 \begin{aligned}
 \dot{\sigma}(t) &= CBu^s(t)\\
 u^s(t) &\in -\alpha\Sgn(\sigma(t)).
\end{aligned}
\right.
 \label{slidingVarDyn}
\end{equation}
The nominal system (\ref{linSyst}) can be rewritten as
\begin{align}
 \dot{x}(t) &= (I-B(CB)^{-1}C)Ax(t) + Bu^s(t),\\
 \intertext{or equivalently}
 \dot{x}(t) &= \Pi Ax(t) + Bu^s(t),
 \label{ProjDyn}
\end{align}
with $\Pi \coloneqq I-B(CB)^{-1}C$.
Two interesting properties of $\Pi$ are $C\Pi = 0$ and $\Pi$ is a projector \cite{edwards1998sliding}.
Taking the integral form of system (\ref{ProjDyn}) yields the relation
\begin{align}
x(t) &= \Phi(t, t_0)x(t_0) + \int_{t_0}^t\!\Phi(t,\tau)Bu^s(\tau)\mathrm{d}\tau,\label{projectedDyn}
\end{align}
with $\Phi(t, t_0) = e^{\Pi A(t-t_0)}$ the state transition matrix for the system~\eqref{ProjDyn}.
Some of the properties of $\Phi$ are given in the following lemma.
\begin{lem}\label{lem:st_mat}
 One has $\dot{\Phi}(t,t_0) = \Pi A\Phi(t,t_0)$, $\Phi(t_0, t_0) = I$,
 and $C\Phi = C$ for all $t\geq t_0$.
\end{lem}
\begin{proof}
 One has $C\dot{\Phi}(t,t_0) = 0$ so $C\Phi(t,t_0) = C\Phi(t_0,t_0) = C$ for all $t\geq t_0$.
\end{proof}

\section{Discrete-time controllers}\label{sec:dt}

\subsection{Classical discretization methods to obtain discrete-time controllers}\label{sec:disc_methods}

From now on, $\bar{u}^{eq}$ and $\bar{u}^s$ are sampled control laws defined as right-continuous step functions:
\begin{align}
 \bar{u}^{eq}(t) &= \bar{u}^{eq}_k,\quad t\in[t_k,t_{k+1})\\
 \bar{u}^{s}(t) &= \bar{u}^{s}_k,\quad t\in[t_k,t_{k+1}).\label{dtU}
\end{align}
The goal of the discretization process is to choose the elements of the sequences
$\{\bar{u}^{eq}_k\}$ and $\{\bar{u}^{s}_k\}$ such that the discrete-time system
exhibits properties as close as possible to the ones with a continuous-time controller.
In continuous time, sliding-mode controlled systems have their evolution divided into two phases: the \emph{reaching phase}, where
$\|\sigma\|>0$ and is decreasing, and the \emph{sliding phase}, where $\sigma=0$ and the sliding motion occurs. It is well known that
the sliding motion does not occur in general in discrete time even on a nominal system because of the error
induced by the discretization. This has led to the definition of quasi-sliding surfaces~\cite{gao1995discrete}.
By analogy with the Filippov's solutions we define the following. 
\begin{de}[Discrete-time sliding phase]\label{de:dtSlidingPhase}
 A system~\eqref{linSyst}, in its sampled-data form, is in the discrete-time sliding phase if $\bar{u}^s$
 takes values in $(-\alpha,\alpha)^p$.
\end{de}
Such a definition appears to be new in the discrete-time sliding mode control field since it implies that the
discrete-time discontinuous controller is itself set-valued, just as its continuous-time counterpart
in \eqref{linSyst} and \eqref{slidingVarDyn}. This will be made possible with an implicit implementation, as proved in
\cite{acary2010implicit} and \cite{acary2012chattering}.
It is crucial not to define the sliding phase in terms of $\bar{\sigma}_{k}$,
but rather in terms of the discontinuous input $\bar{u}^s$. In continuous-time, Definition~\ref{de:dtSlidingPhase}
applied to $u^s$ implies that the system is in the sliding phase $\sigma\equiv0$, see~\eqref{slidingVarDyn}.
It holds with a matched perturbation as long as for all $t\geq0$, $\|\xi(t)\|_{\infty}\leq\delta<\alpha$
for some $\delta\geq0$. 

Integrating the nominal version of \eqref{linSyst} over $[t_k, t_{k+1})$
and using the expressions in \eqref{dtU}, we obtain the ZOH discretization of the system:
\begin{align}
 \bar{x}_{k+1} &= e^{Ah}\bar{x}_k + B^{*}\bar{u}^{eq}_k + B^{*}\bar{u}^s_k,\label{eq:ZOH}
 \intertext{with}
 B^{*} &\coloneqq \int_{t_k}^{t_{k+1}}\!e^{A(t_{k+1}-\tau)}B\mathrm{d}\tau.\label{eq:Bstar}
\end{align}
Let $\Psi\coloneqq \int_{t_k}^{t_{k+1}}\!e^{A(t_{k+1}-\tau)}\mathrm{d}\tau=\sum_{l=0}^{\infty}\frac{A^lh^{l+1}}{(l+1)!}$,
then $B^{*} = \Psi B$.
We now present different choices for the values $\bar{u}^{eq}_k$ and $\bar{u}^s_k$. Firstly, standard methods
are described, while the new method is studied in the next section.
Here $\bar{u}^{eq}_k$ and $\bar{u}^s_k$ are the discretized values of the continuous-time control law $u^{eq}$ and $u^{s}$.
From all the possible time-discretization schemes, we focus on the one-step explicit, implicit, and midpoint ones.
With the expressions found for $u^{eq}$ and $u^{s}$ in
\eqref{defUeq} and \eqref{slidingVarDyn}, the proposed discretized values for the equivalent control $\bar{u}^{eq}_k$ are:
\begin{subequations}
 \label{eqs:Ueq}
 \begin{align}
 \bar{u}^{eq}_{k,e} &= -(CB)^{-1}CA\bar{x}_k &&\text{explicit input,}\label{UeqZOHe}\\
 \bar{u}^{eq}_{k,i} &= -(CB)^{-1}CA\bar{x}_{k+1} &&\text{implicit input,}\label{UeqZOHi}\\
 \bar{u}^{eq}_{k,m} &= 1/2(\bar{u}^{eq}_{k,e}+\bar{u}^{eq}_{k,i}) &&\text{midpoint input,}\label{UeqZOHm}
\end{align}
\end{subequations}
and the two possibilities for the discontinuous control $\bar{u}^s_k$ are:
\begin{subequations}
 \label{eqs:Us}
 \begin{align}
  \bar{u}^s_k &=-\alpha\sgn(\bar{\sigma}_k)  &&\text{explicit input,}\label{eq:Use}\\
  \bar{u}^s_k &\in-\alpha\Sgn(\bar{\sigma}_{k+1})  &&\text{implicit input.}\label{eq:Usi}
 \end{align}
\end{subequations}
We use the singled-valued $\sgn$ function in \eqref{eq:Use} since the case
$\bar{\sigma}_k=0$ is not worth considering for explicit inputs. Moreover with the set-valued $\Sgn$
function, if $\bar{\sigma}_k = 0$,
then we would have $\Sgn(\bar{\sigma}_k)\in[-\alpha,\alpha]^p$ and there is no proper selection procedure to get
a value for $\bar{u}^s_k$. In the next subsection, we present the selection procedure in the implicit case.
The objective in Section~\ref{sec:sim} is to study the behaviour of the closed-loop system when different
combinations of \cref{UeqZOHe,UeqZOHi,UeqZOHm,eq:Use,eq:Usi} are used.
The most commonly used control law is the combination of \eqref{UeqZOHe} and \eqref{eq:Use}.
This kind of discretization has been studied in \cite{galias2006complex,galias2008analysis,wang2008zoh},
with a focus on the sequence formed by $\bar{\sigma}_k$ once the system state approaches the sliding manifold.
The implicit discretization~\eqref{eq:Usi} was first introduced in \cite{acary2010implicit}
and \cite{acary2012chattering}. In the following, we provide more details on it.

\subsection{Definition and properties of the implicitly discretized discontinuous control input}

With the implicit method~\eqref{eq:Usi}, for each $k\in\N$, $\bar{u}^s_k$ is computed as the solution to the
generalized equation
\begin{equation}
 \begin{cases}
  \widetilde{\sigma}_{k+1} = \bar{\sigma}_k + CB^{*}\bar{u}^s_k\\
  \bar{u}^s_k\in-\alpha\Sgn(\widetilde{\sigma}_{k+1}).
 \end{cases}
 \label{lcp:inexact}
\end{equation}
As shown in \cite{acary2010implicit}, there exists $k_0\in\N$ such that for all $k>k_0$,
$\widetilde{\sigma}_k=0$.
Let us write the discrete-time system with an implicit discretization of $u^s$ and let $\bar{u}^{eq}_k$
be computed using one of the methods in \cref{UeqZOHe,UeqZOHi,UeqZOHm}:
\begin{equation}
 \begin{cases}
  \bar{x}_{k+1} = e^{Ah}\bar{x}_k + B^{*}\bar{u}^{eq}_k + B^{*}\bar{u}^s_k\\
  \widetilde{\sigma}_{k+1} = C\bar{x}_k + CB^{*}\bar{u}^s_k\\
  \bar{u}^s_k\in-\alpha\Sgn(\widetilde{\sigma}_{k+1}).
 \end{cases}
 \label{eq:genDTsys}
\end{equation}
Nothing guarantees that $C(e^{Ah}\bar{x}_k + B^{*}\bar{u}^{eq}_k) = C\bar{x}_k$. Hence,
$\widetilde{\sigma}_{k+1}$ is in general different from $\bar{\sigma}_{k+1}$ because of the discretization
error on $u^{eq}$. Indeed this error, stemming from the discretization of $u^{eq}$, can be seen as a
perturbation of the closed-loop system, even if $\xi \equiv 0$.
Therefore, $\widetilde{\sigma}_{k+1}$ can be considered as an approximation of $\bar{\sigma}_{k+1}$.

The system~\eqref{lcp:inexact} can be analysed using the Affine Variational Inequality (AVI)
formalism~\cite{facchinei2003finite}.
Let $N_{[-\alpha,\alpha]^p}(\lambda)$ be the normal cone to the box $[-\alpha,\alpha]^p$ at $\lambda$,
that is $N_{[-\alpha,\alpha]^p}(\lambda) = \{d\in\R^p\mid \langle d, y-\lambda\rangle \leq0, \forall y\in[-\alpha, \alpha]^p\}$.
The relation $\bar{u}_k^s\in-\Sgn(\widetilde{\sigma}_{k+1}) \Longleftrightarrow \widetilde{\sigma}_{k+1}\in-N_{[-\alpha,\alpha]^p}(\bar{u}_k^s)$
enables us to 
transform~\eqref{lcp:inexact} into the inclusion:
\begin{equation}
 0 \in \bar{\sigma}_k + CB^{*}\bar{u}_k^s + N_{[-\alpha,\alpha]^p}(\bar{u}_k^s).
 \label{eq:incl_avi}
\end{equation}
The inclusion~\eqref{eq:incl_avi} is satisfied if and only if $\bar{u}_k^s$ is the solution
of the AVI:
\begin{equation}
 \text{Find }z\in[-\alpha,\alpha]^p \text{ such that}\quad
 (y-z)^T(\bar{\sigma}_k + CB^{*}z)\geq0, \quad \forall y\in[-\alpha,\alpha]^p.
 \label{eq:avi}
\end{equation}
Let $\mathrm{SOL}(CB^*, C\bar{x}_k)$ denote the set of all solutions to the AVI~\eqref{eq:avi}.
The existence and uniqueness of solutions to this AVI are now presented.
\begin{lem}\label{lem:existence}
 The AVI~\eqref{eq:avi} has always a solution.
\end{lem}
\begin{proof}
 Since the mapping $z\mapsto CB^{*}z+\bar{\sigma}_k$ is continuous, we can apply the Corollary~2.2.5, p.~148
 in \cite{facchinei2003finite}.
\end{proof}
\begin{lem}\label{lem:uniqueness}
 The AVI~\eqref{eq:avi} has a unique solution for all $\bar{\sigma}_k\in\R^n$ if and only if $CB^{*}$ is a $\mathbf{P}$-matrix.
 \end{lem}
\begin{proof}
 In \cite{facchinei2003finite}, using Theorem~4.3.2 p.~372 and Example~4.2.9 p.~361 yields the result.
\end{proof}
In most ECB-SMC systems, $CB^{*}>0$, therefore is a $\mathbf{P}$-matrix. This approach enables us to analyse
a large class of systems, compared to previous approaches where it is supposed that $CB^{*}$ is scalar
\cite{golo2000robust}. The solution is a function of $\bar{\sigma}_k$ (hence $\bar{x}_k$) and if $CB^{*}$
is a $\mathbf{P}$-matrix, the solution map $C\bar{x}_k\mapsto \bar{u}_k^s = \mathrm{SOL}(CB^*, \bar{\sigma}_k)$
is Lipschitz continuous. 
When the control is scalar or if $CB^{*}$ is diagonal, a solution to \eqref{lcp:inexact}
can be computed as a simple orthogonal projection:
$\bar{u}^s_k = -\proj_{[-\alpha, \alpha]^p}((CB^{*})^{-1}\bar{\sigma}_k)$.
Otherwise a quadratic problem with bounded constraints may be considered \cite{acary2012chattering}. More details on the
numerical aspects and solvers for this kind of problems can be found in \cite{cottle2009linear} and \cite{acary2008numerical}.
Now that we have discussed the existence, uniqueness, some properties of solutions and methods to compute them,
we turn our attention to the performance of each controller.
\section{Discretization performance}\label{sec:perf}

\subsection{Discretization of the state-continuous control}

Let us focus on the discretization error on $u^{eq}$ and more specifically on its effect on
the sliding variable. In other words we analyse how the invariance property in \eqref{defUeq} is preserved
after discretization. In the following, $\bar{u}^s$ is set to $0$.
Let $\Delta\bar{\sigma}_{k} \coloneqq \bar{\sigma}_{k+1}-\bar{\sigma}_k$ be the local variation of the
sliding variable due to the discretization error on $u^{eq}$.

\subsubsection{Explicit discretization}

With an explicit discretization of $u^{eq}$ as in \eqref{UeqZOHe}, and using \eqref{eq:ZOH} the closed-loop discrete-time system dynamics is
\begin{equation}
 \bar{x}_{k+1} = \Phi_k^e\bar{x}_k,\label{eq:ZOHeState}
\end{equation}
with $\Phi_k^e \coloneqq e^{Ah} - \Psi \Pi_B A$ and $\Pi_B \coloneqq B(CB)^{-1}C = I-\Pi$.
\begin{lem}\label{lem:errDiscUeqE}
 With an explicit discretization of $u^{eq}$, the discretization error $\Delta\bar{\sigma}_{k}$
 is of order $\mathcal{O}(h^2)$.
\end{lem}

\begin{proof}
 Starting from~\eqref{eq:ZOHeState}, one obtains
 \begin{align}
 \bar{x}_{k+1} - \bar{x}_k &= (e^{Ah} - I)\bar{x}_k - \Psi\Pi_BA\bar{x}_k.\label{eq:fullZOHe}\\
 \intertext{Hence, using the definition of $\Psi$ in Section~\ref{sec:disc_methods}:}
 \Delta\bar{\sigma}_k &= \bar{\sigma}_{k+1}-\bar{\sigma}_k = C\left(e^{Ah}-I - \Psi\Pi_BA\right)\bar{x}_k\label{eq:deltaSigmaE}\\
 &= C\left(Ah + \frac{A^2h^2}{2} - \left(h+\frac{Ah^2}{2}\right)\Pi_BA + \mathcal{O}(h^3)\right)\bar{x}_k\\
 &= C\left(\left( h + \frac{Ah^2}{2}\right)(I-\Pi_B) A\right)\bar{x}_k + \mathcal{O}(h^3)\\
 &= hC\Pi A\bar{x}_k + \frac{h^2}{2}CA\Pi A\bar{x}_k + \mathcal{O}(h^3)\\
 &= \frac{h^2}{2}CA\Pi A\bar{x}_k + \mathcal{O}(h^3).\label{errZOHeUeq}
\end{align}
The action of $u^{eq}_{k,e}$ does not keep $\bar{\sigma}$ constant and the error is of order $\mathcal{O}(h^2)$.
\end{proof}

\subsubsection{Implicit discretization}
The recurrence equation \eqref{eq:ZOH} combined with \eqref{UeqZOHi} yields
\begin{align}
 \bar{x}_{k+1} &= e^{Ah}\bar{x}_k - \Psi\Pi_B A\bar{x}_{k+1},\label{eq:fullZOHi}\\
 \intertext{that is:}
 \bar{x}_{k+1} &= W^{-1}e^{Ah}\bar{x}_k,
\end{align}
with $W = I + \Psi\Pi_B A$.
\begin{lem}\label{lem:order_impl_eq}
 With an implicit discretization of $u^{eq}$, the discretization error $\Delta\bar{\sigma}_{k}$
 is of order $\mathcal{O}(h^2)$.
\end{lem}
\begin{proof}
 There exists a Taylor expansion for $W^{-1}$ if $\Psi\Pi_B A$ has all its eigenvalues in the unit disk.
 Since $\Psi\to0$ as $h\to0$, it is always possible to find an $h_0$ such that this condition holds
 for all $h_0>h>0$. Since we are interested in an asymptotic property, such restriction on $h$ does not
 play any role.
 Let us compute the finite expansion of $W^{-1}e^{Ah}$:
\begin{align}
 W^{-1}e^{Ah} &= \left( I - \Psi\Pi_BA + (\Psi\Pi_BA)^2 \right)\left( I + Ah + \frac{A^2h^2}{2} \right) + \mathcal{O}(h^3)\\
 &= I - \Psi\Pi_BA  + Ah + (\Psi\Pi_BA)^2  - h\Psi\Pi_BA^2 + \frac{A^2h^2}{2} + \mathcal{O}(h^3)\\
 &= I - (\Pi_BA - A)h + \left(-\frac{\Pi_BA\Pi_BA}{2} + \Pi_BA\Pi_BA - \Pi_BA^2 + \frac{A^2}{2}\right)h^2 + \mathcal{O}(h^3)\\
 &= I -(\Pi_BA - A)h + \left(\frac{\Pi_BA\Pi_BA}{2} - \Pi_BA^2 + \frac{A^2}{2}\right)h^2 + \mathcal{O}(h^3).
\end{align}
Then the variation of the sliding variable is
\begin{align}
 \Delta\bar{\sigma}_k &= C(W^{-1}e^{Ah}-I)\bar{x}_k = \label{eq:deltaSigmaI}
 h(-A + A)\bar{x}_k + h^2\left(\frac{CA\Pi_BA}{2} - \frac{CA^2}{2}\right)\bar{x}_k + \mathcal{O}(h^3)\\
 &= \frac{h^2}{2}CA(I-\Pi_B)A\bar{x}_k + \mathcal{O}(h^3)\\
 &= -\frac{h^2}{2}CA\Pi A\bar{x}_k + \mathcal{O}(h^3).\label{Oh2impl}
\end{align}
\end{proof}
\begin{lem}\label{lem:midpoint}
 With a midpoint method \eqref{UeqZOHm}, the error is of order $\mathcal{O}(h^3)$.
\end{lem}
\begin{proof}
 With the midpoint method \eqref{UeqZOHm}, the recurrence equation is
 \begin{align}
  \bar{x}_{k+1} &= 1/2(e^{Ah}\bar{x}_k -\Psi\Pi_BA\bar{x}_{k}) +1/2(e^{Ah}\bar{x}_k -\Psi\Pi_BA\bar{x}_{k+1}).
 \end{align}
  The two terms are the right-hand side in~\eqref{eq:fullZOHe} and \eqref{eq:fullZOHi}. Then the
  discretization error is the mean of the discretization error of the explicit and implicit case.
  Since the first term in \eqref{Oh2impl} is the opposite of the first one in \eqref{errZOHeUeq},
  the term in $h^2$ vanishes and the error is of order $\mathcal{O}(h^3)$.
\end{proof}

\subsection{Discretization of both control inputs}
In the following, we consider the sliding variable dynamics with the state-continuous and
discontinuous control. It is expected that $\bar{\sigma}$ goes to $0$ and once it reaches zero,
stays at this value. The proposed metric to measure the performance of the discrete-time controller
is the Euclidean norm of the sliding variable when the system state is close to the sliding manifold.
Let $\varepsilon_k \coloneqq \|\bar{\sigma}_{k+1}\|$ be the discretization error when $\|\bar{\sigma}_{k}\|$ is
small enough.

\subsubsection{Explicit discretization}\label{sec:fullExpl}
In the sliding mode literature, several proposals (seven of them are listed in \cite{furuta2002discrete})
have been made to analyse the behaviour of the closed-loop system near the sliding manifold and
to propose new variable structure control strategies. Despite this,
it is still difficult to analyse the behaviour near the sliding manifold, besides stability. Thus we only
study the invariance of a close neighborhood of the sliding manifold, also to provide an estimate of the
chattering due to the discrete discontinuous control.
\begin{lem}
 Let the closed-loop system state in~\eqref{eq:ZOH} be in an $\mathcal{O}(h^2)$-neighborhood of the sliding manifold at $t=t_k$,
 i.e. $\bar{\sigma}_{k}=\mathcal{O}(h^2)$, but with $\bar{\sigma}_k\neq0$.
 If the discontinuous part $u^s$ of the control is discretized using the explicit scheme~\eqref{eq:Use}, then the
 discretization error $\varepsilon_k$ is of order $\mathcal{O}(h)$ and the system exits
 the $\mathcal{O}(h^2)$-neighborhood.
\end{lem}
\begin{proof}
 Starting from equation \eqref{eq:ZOH} and using the control inputs $\bar{u}^{eq}_k = -(CB^{*})^{-1}CA\bar{x}_k$
and $\bar{u}^s_k = -\alpha\sgn(C\bar{x}_k)$, we have the following:
\begin{align}
 \bar{\sigma}_{k+1} &= C(e^{Ah} - \Psi\Pi_BA)\bar{x}_k - CB^{*}\sgn(C\bar{x}_k)
 \intertext{that is}
 \bar{\sigma}_{k+1} &= \bar{\sigma}_k + \Delta_k - CB^{*}\sgn(\bar{\sigma}_k),\label{RecSigmaE}
 \intertext{with}
 \Delta_k &\coloneqq C(e^{Ah} - I - \Psi\Pi_BA)\bar{x}_k.\label{eq:Delta_k}
\end{align}
Let us study the square of the norm of the sliding variable:
\begin{align}
 \bar{\sigma}_{k+1}^T\bar{\sigma}_{k+1} = \|\bar{\sigma}_k\|^2 + \|\Delta_k\|^2 + \|CB^{*}\sgn(\bar{\sigma}_k)\|^2
 + \bar{\sigma}_k^T\Delta_k - \bar{\sigma}_k^TCB^{*}\sgn(\bar{\sigma}_k) - \Delta_k^TCB^{*}\sgn(\bar{\sigma}_k).\label{eq:SigDTe}
\end{align}
From Lemma~\ref{lem:errDiscUeqE}, we have $\|\Delta_k\|^2 = \mathcal{O}(h^4)$.
For each other term, we can compute its order with respect to $h$:
\begin{align}
 \|CB^{*}\alpha\sgn(\bar{\sigma}_k)\|^2 &= \left\|\sum_{k=0}^{\infty} h^{l+1}\frac{CA^lB}{(l+1)!}\alpha\sgn(\bar{\sigma}_k)\right\|^2\notag{}\\
 &\leq \|hCB\alpha\sgn(\bar{\sigma}_k)\|^2 + \mathcal{O}(h^3)\label{eq:chat_sgn}\\
 \Delta_k^TCB^{*}\alpha\sgn(\bar{\sigma}_k) &= \mathcal{O}(h^3).
\end{align}
Using the Cauchy-Schwarz inequality on the remaining terms yields
\begin{align}
 |\bar{\sigma}_k^T\Delta_k| &\leq \|\bar{\sigma}_k\| \|\Delta_k\|\\
 |\bar{\sigma}_k^TCB^{*}\sgn(\bar{\sigma}_k)| &\leq \|\bar{\sigma}_k\| \|CB^{*}\sgn(\bar{\sigma}_k)\|.
\end{align}
If $\|\bar{\sigma}_k\| = \mathcal{O}(h^2)$, then the above terms are of order $\mathcal{O}(h^4)$ and $\mathcal{O}(h^3)$.
Hence, the dominant term in \eqref{eq:SigDTe} will be $\|hCB\alpha\sgn(\bar{\sigma}_k)\|^2$. Let $\{\lambda_i\}$ be the spectrum
of $hCB$, with $\lambda_m = \min_i |\lambda_i|$ and $\lambda_M = \max_i |\lambda_i|$. We have the following:
\begin{equation}
 \lambda_m h\alpha\sqrt{p} \leq \|hCB\sgn(\bar{\sigma}_k)\| \leq \lambda_Mh\alpha\sqrt{p}.
\end{equation}
Inserting this in~\eqref{eq:SigDTe} yields that $\|\bar{\sigma}_{k+1}\|$ has order $\mathcal{O}(h)$.
\end{proof}
%
Therefore with an explicit discretization of $u^s$, the main error comes
from the discretization of the discontinuous control $u^s$, since it increases the error by an order $h$.

\subsubsection{Implicit discretization}

In the following, we are interested in studying the discretization error in the same context as for
the previous lemma.
\begin{lem}
 Let the closed-loop system be
 in the discrete-time sliding phase, as defined in Definition~\ref{de:dtSlidingPhase}.
 If the discontinuous part $u^s$ of the control is discretized using an implicit scheme,
 then the discretization error $\varepsilon_{k}$ has the same order as the discretization
 error $\Delta\bar{\sigma}_k$ on $u^{eq}$. That is $\mathcal{O}(h^2)$ for the methods~\eqref{UeqZOHe}
 and~\eqref{UeqZOHi}, and $\mathcal{O}(h^3)$ for the midpoint method~\eqref{UeqZOHm}.
\end{lem}
\begin{proof}
 Let $\Delta_k=C(e^{Ah}-I)x_k + C\Psi B\bar{u}^{eq}_k$, with $\bar{u}^{eq}_k$ any method in~\eqref{UeqZOHe}-\eqref{UeqZOHm}.
 The system is supposed to be in the discrete-time sliding phase, that is $u^{s}_{k}\in(-\alpha,\alpha)^p$.
 Then $\widetilde{\sigma}_{k+1} = \bar{\sigma}_k + CB^{*}\bar{u}_k^s = 0$. From~\eqref{eq:ZOH} one has:
\begin{align}
 \bar{\sigma}_{k+1} &= \bar{\sigma}_k + \Delta_k + CB^{*}\bar{u}_k^s\label{eq:RecSigmaI}= \Delta_k
\end{align}
Let us go through all the possible discretization of $u^{eq}$.
In the explicit case~\eqref{UeqZOHe}, $\Delta_k$ is the quantity studied in~\eqref{eq:deltaSigmaE}-\eqref{errZOHeUeq}
and is of order $\mathcal{O}(h^2)$. With the implicit method~\eqref{UeqZOHi}, the order of $\Delta_k$ is $\mathcal{O}(h^2)$,
as shown in~\eqref{eq:deltaSigmaI}-\eqref{Oh2impl}. With the midpoint method~\eqref{UeqZOHm}, going along the lines of
the proof of Lemma~\ref{lem:midpoint}, we get that $\Delta_k$ is of order $\mathcal{O}(h^3)$.
\end{proof}
\begin{remark}
 In Section~\ref{sec:stab_analysis}, conditions are derived to ensure that
 the system stays in the discrete-time sliding phase once it reaches it, with or without perturbation.
\end{remark}

\subsection{Disturbance and chattering}


We consider the system~\eqref{linSyst} with a non-zero perturbation. In the following we consider only
continuous perturbations.
Let us first recall some facts with a continuous-time controller.
The design procedure for the equivalent part is the same as in the unperturbed case.
If the system is in the sliding phase, it can stay on the sliding surface
despite the perturbation $\xi(t)$, if the origin is contained in the set
obtained using the Filippov's convexification procedure.
We suppose here that $CB$ is nonsingular. Then the condition for rejection can be reformulated
as $0\in-\alpha\Sgn(0)+\xi(t)$.
The condition $\alpha>\|\xi\|_{\infty}$ ensures the unconditional existence of the sliding motion.

We now investigate the performance of the discretization scheme with respect to the perturbation.
First we analyse how the perturbation affects the discrete-time dynamics of the system.
To take it into account, we just need to add a term
$p_k\coloneqq \int_{t_k}^{t_{k+1}}\!e^{A(t_{k+1}-\tau)}B\xi(\tau)\mathrm{d}\tau$ to the recurrence
relation~\eqref{eq:ZOH}. This yields:
\begin{equation}
 \bar{x}_{k+1} = e^{Ah}\bar{x}_k + B^{*}\bar{u}^{eq}_k + B^{*}\bar{u}^s_k+p_k.
\label{eq:ZOHp}
\end{equation}
Note that with the type of perturbation we consider, $p_k$ is of
order $\mathcal{O}(h^2)$.

\subsubsection{Explicit case}\label{sec:dist_chattering_explicit}

Given the expression in~\eqref{eq:ZOHp}, to take into account the perturbation,
we just have to add the term $Cp_k$ to~\eqref{RecSigmaE}:
\begin{equation}
 \bar{\sigma}_{k+1} = \bar{\sigma}_k + \Delta_k - CB^{*}\sgn(\bar{\sigma}_k) + Cp_k.
 \label{eq:SigDTep}
\end{equation}
Using the same analysis as in the previous subsection, that is computing $\|\bar{\sigma}_{k+1}\|^2$, we have all the terms
in~\eqref{eq:SigDTe}, plus the following:
\begin{align}
 |\bar{\sigma}_k^TCp_k| \leq \|\bar{\sigma}_k\|\|Cp_k\| &= \mathcal{O}(h^3)\quad\text{by Cauchy-Schwarz inequality}\\
 |\Delta_k^TCp_k| \leq \|\Delta_k\|\|Cp_k\|&= \mathcal{O}(h^3)\quad\text{by Cauchy-Schwarz inequality}\\
 (CB^{*}\alpha\sgn(\bar{\sigma}_k))^TCp_k &=  \mathcal{O}(h^2)\label{eq:chat_cross} \\
 \|Cp_k\|^2 &= \mathcal{O}(h^2).\label{eq:chat_pert}
\end{align}
Thus the dominant terms in \eqref{eq:SigDTep} are of order $h$ and are:
$\|Cp_k\|$, $\|CB^{*}\alpha\sgn(\bar{\sigma}_k)\|$ and $(CB^{*}\alpha\sgn(\bar{\sigma}_k))^TCp_k$.
Those terms induce chattering and they all have the same order
with respect to the timestep $h$.

\subsubsection{Implicit case}

Updating~\eqref{eq:RecSigmaI} to take into account the perturbation yields:
\begin{align}
 \bar{\sigma}_{k+1} &= \bar{\sigma}_k + \Delta_k - CB^{*}\sgn(\bar{\sigma}_{k+1}) + Cp_k\\
 &= \Delta_k + Cp_k.
\end{align}
Recall that $\Delta_k$ has order $\mathcal{O}(h^2)$. Hence the chattering due to the perturbation
will be predominant. In the next Section, we present a method for computing $\bar{u}^{eq}_k$
such that the discretization error $\Delta_k=0$. Then the chattering is solely due to the perturbation.

\section{Exact discrete equivalent control}\label{sec:design}

Let us propose a new control scheme for a discrete-time LTI plant using sliding mode control.
Its derivation is along the same lines as in Section~\ref{sec:recall}, that is we first design the equivalent
control $\bar{u}^{eq}$ and then the discontinuous part $\bar{u}^s$.

As showed in \eqref{slidingVarDyn}, $u^{eq}$ is defined such that the dynamics of the sliding variable depends
only on the input $u^s$. Starting from \eqref{eq:ZOH} and left multiplying by $C$, one obtains:
\begin{equation}
 C\bar{x}_{k+1} = Ce^{Ah}\bar{x}_k +  CB^{*}\bar{u}^{eq}_k + CB^{*}\bar{u}^s_k.\\
 \label{sigmaZOH}
\end{equation}
Using \eqref{projectedDyn} with $t = t_{k+1}$ and $t_0 = t_k$, we obtain
\begin{equation}
 \sigma(t_{k+1}) = C\Phi(t_{k+1}, t_k)x(t_k) + C\int_{t_k}^{t_{k+1}}\!\Phi(t_{k+1},\tau)Bu^s(\tau)\mathrm{d}\tau.
 \label{sigmaCT}
\end{equation}
Our goal is to have $C\bar{x}_{k+1} = Cx(t_{k+1})$ if $x(t_k) = \bar{x}_k$ and both $u^s$ and $\bar{u}^s$ set to $0$.
Then setting the last term of \eqref{sigmaZOH} and \eqref{sigmaCT} to $0$ and using Lemma~\ref{lem:st_mat},
 the following condition holds: 
\begin{align}
 C\Phi(t_{k+1},t_k) x(t_k) &= Ce^{Ah}\bar{x}_k +  CB^{*}\bar{u}^{eq}_k,
 \intertext{that is}
 CB^{*}\bar{u}^{eq}_k &= C(I-e^{Ah})\bar{x}_k.
 \label{UeqZOH}
\end{align}
In \cite{furuta1990sliding}, this expression for the equivalent control was already derived, when the sliding
variable is scalar. In \cite{utkin1994sliding}, using a deadbeat-like approach, a term similar to \eqref{UeqZOH} can
also be found.
If we substitute this expression for $\bar{u}^{eq}_k$ in \eqref{sigmaZOH}, then, as expected, we obtain
\begin{equation}
 \bar{\sigma}_{k+1} = \bar{\sigma}_k + CB^{*}\bar{u}^s_k,
 \label{recUs}
\end{equation}
which is the discrete counterpart of \eqref{slidingVarDyn}. For the design of $\bar{u}^s$, let us choose $\bar{u}^s_k$ such that $\bar{u}^s$ steers $\bar{\sigma}_{k}$ to $0$ in finite
time. Following the work in \cite{acary2010implicit} and
\cite{acary2012chattering}, we use an implicit discretization of the continuous-time control law.
The discrete-time sliding variable dynamics is given by \eqref{recUs}
and $\bar{u}^s_k \in -\alpha\Sgn(\bar{\sigma}_{k+1})$. Inserting \eqref{UeqZOH} in \eqref{eq:ZOH},
the discrete-time dynamics of the nominal controlled plant is
\begin{equation*}
\begin{cases}
 \bar{x}_{k+1} = (e^{Ah} + B^{*}(CB^{*})^{-1}C(I-e^{Ah}))\bar{x}_k + B^{*}\bar{u}^s_k\notag\\
 \bar{\sigma}_{k+1} = \bar{\sigma}_k + CB^{*}\bar{u}^s_k.\label{slidingVarRec}
\end{cases}
\end{equation*}
Using the framework of generalized (set-valued) equations, the discrete-time sliding variable
dynamics is
\begin{equation}
 \begin{cases}
  \bar{\sigma}_{k+1} = \bar{\sigma}_k + CB^{*}\bar{u}^s_{k}\\
  \bar{u}^s_{k} \in -\alpha\Sgn(\bar{\sigma}_{k+1}).
 \end{cases}
 \label{DTMMsys}
\end{equation}
This system has the same structure as in~\eqref{lcp:inexact}, although with the important difference that
we have here $\widetilde{\sigma}_{k+1} = \bar{\sigma}_{k+1}$. Therefore, we can show that
$\bar{\sigma}_{k}$ goes to $0$ in finite time, see Proposition~\ref{prop:ft} in Section~\ref{sec:stab_analysis}.
Hence in the nominal case, the system reaches the sliding surface at a certain
time $t_{k_0}$ and then for all $k>k_0,\; \bar{\sigma}_{k} = 0$.
\begin{lem}
 Suppose $CB^{*}$ is a $\mathbf{P}$-matrix. Then the only equilibrium pair of the system~\eqref{DTMMsys}
 is $(\bar{\sigma}_*, \bar{u}^s_*) = (0,0)$.
\end{lem}
\begin{proof}
A pair $(\bar{\sigma}, \bar{u}^s)$ is an equilibrium of \eqref{DTMMsys} if and only if $CB^{*}\bar{u}^s = 0$.
If $CB^{*}$ is a $\mathbf{P}$-matrix, then it has full-rank and $CB^{*}\bar{u}^s = 0$ is equivalent
to $\bar{u}^s = 0$. By the definition of the $\Sgn$ multifunction in~\eqref{eq:def_Sgn},
this is only possible if $\bar{\sigma} = 0$.
\end{proof}
With this scheme the two control inputs are
\begin{equation}
 \begin{cases}
  \bar{u}^{eq}_k = \left(CB^{*}\right)^{-1}C(I-e^{Ah})\bar{x}_k\\
  \bar{u}^{s}_k\qquad\text{solution of \eqref{DTMMsys}}.
 \end{cases}
 \label{eq:new_control}
\end{equation}
This controller is nonanticipative since $\bar{u}^{eq}_k$ depends only on the model parameters and $\bar{x}_k$.
Moreover $\bar{u}^s_{k}$ is the unique solution to \eqref{DTMMsys} given that $CB^{*}>0$, using similar arguments
as in Lemma~\ref{lem:uniqueness}.
This controller retains the structure of the continuous-time sliding mode controller. It is different from the approach
that can be found in \cite{utkin1994sliding} or \cite{lin2004total} since the equivalent part $u^{eq}$ is not
chosen as the solution to a deadbeat control problem. As a result, the magnitude of the control is of order
$\mathcal{O}(1)$ with respect to the timestep $h$, whereas it is of order $\mathcal{O}(h^{-1})$ in the deadbeat case,
see \cite{lin2004total}.

\section{Simulations of a 2-dimensional system}\label{sec:sim}

To illustrate the results obtained with different discretization methods,
let us simulate the following controlled system: 
\begin{equation}
 \begin{cases}
  \dot{x}(t) = Ax(t)+B\bar{u}(t)\\
  \sigma = Cx\\
  \bar{u}(t) = \bar{u}^{eq}(t) + \bar{u}^s(t)
 \end{cases}
 \label{sys:2Dunstable}
 \left.
 \begin{aligned}
  A &= \begin{pmatrix}
0& 1\\19&-2\\
\end{pmatrix},\\
  B &= \begin{pmatrix}0\\1 \end{pmatrix},\;
  C^T = \begin{pmatrix}1\\1\end{pmatrix}.
 \end{aligned}\right.
\end{equation}
The matrix A has the eigenvalues $\lambda_1=3.47$ and $\lambda_2=-5.47$. The dynamics on the
sliding surface is given by $\Pi A = \begin{pmatrix}0&1\\0&-1\end{pmatrix}$, which has eigenvalues $0$ and $-1$.
Through this section, we chose $\alpha=1$.
The initial state is $(-15, 20)^T$. The first set of simulations uses
a timestep $0.3$~s for the control and the second one a timestep $0.03$~s.
The simulations run for $150$ s and were carried out with the \textsc{siconos} software
package \cite{acary2007introduction}\footnote{\url{http://siconos.gforge.inria.fr}}.
Figures were created using Matplotlib \cite{hunter2007matplotlib}.
The schemes presented in \cref{UeqZOHe,UeqZOHi,UeqZOHm} are used, as well as the two
schemes in \cref{eq:Use,eq:Usi} for the discretization of $u^s$, on the ZOH sampled-data version of
the system~\eqref{sys:2Dunstable}. In Subsection~\ref{sec:simNom}, the nominal system~\eqref{sys:2Dunstable} is
simulated and in Subsection~\ref{sec:pert} a matching perturbation is added. For each set of simulations, three types of
figure are shown. The first one
present an overview of the trajectories of the different closed-loop systems (like Fig.~\ref{fig:sim1} and~\ref{fig:sim2}).
The next one displays some details, around the origin (Fig.~\ref{fig:sim1_zoom}, \ref{fig:sim2_zoom} and~\ref{fig:sim3_zoom}).
Finally, we present plots of the different discontinuous inputs
(Fig.~\ref{fig:sim1_us}, \ref{fig:sim2_us} and~\ref{fig:sim3_us}). Markers are also added to help visualize the position
of the closed-loop system at some of the time instants $t_k$.

\subsection{Nominal case}\label{sec:simNom}
\begin{figure}[ht]
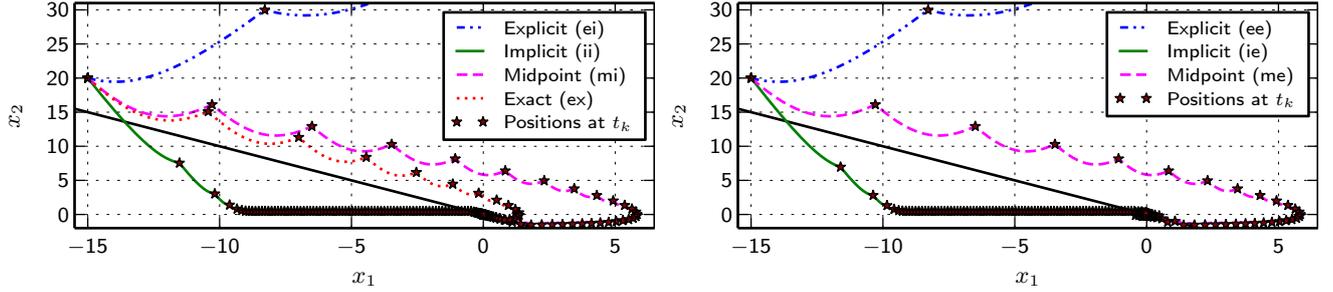

 \centering
 \begin{subfigure}[b]{\lenSubFig\linewidth}
   \centering
   \beginpgfgraphicnamed{2D-implicit-h-0.3}\input{pgf/2D-implicit-h-0.3.pgf}\endpgfgraphicnamed
   \caption{Implicit discretization of $u^s$. (ei) is for pair \eqref{UeqZOHe},
   \eqref{eq:Usi}; (ii) for \eqref{UeqZOHi}, \eqref{eq:Usi};
   (mi) for \eqref{UeqZOHm}, \eqref{eq:Usi}; (ex) for \eqref{eq:new_control}.}
   \label{fig:sim1:i}
 \end{subfigure}
 \begin{subfigure}[b]{\lenSubFig\linewidth}
  \centering
  \beginpgfgraphicnamed{2D-explicit-h-0.3}\input{pgf/2D-explicit-h-0.3.pgf}\endpgfgraphicnamed
  \caption{Explicit discretization of $u^s$. (ee) is for pair \eqref{UeqZOHe},
  \eqref{eq:Use}; (ie) for \eqref{UeqZOHi}, \eqref{eq:Use};
  (me) for \eqref{UeqZOHm}, \eqref{eq:Use}.}
  \label{fig:sim1:e}
 \end{subfigure}
 \caption{Simulations of system~\eqref{sys:2Dunstable} using different discretization methods,
 with $h = 0.3$ s and $\alpha=1$.}
 \label{fig:sim1}
\end{figure}
The trajectories for the different closed-loop systems are plotted in Fig.~\ref{fig:sim1}.
The motion in the reaching phase depends
only on the discretization method used for the equivalent control $u^{eq}$.
It is only near the sliding manifold that the discretization method of the discontinuous control $u^s$ plays a role.
If the explicit scheme in \eqref{UeqZOHe} is used for the discretization of $u^{eq}$, the system diverges
(Fig.~\ref{fig:sim1:i} and \ref{fig:sim1:e}, curves (ei) and (ee)). This sampling method can destabilize a system
which is stable in continuous time.
If the implicit scheme in \eqref{UeqZOHi} is used for the discretization of $u^{eq}$, then
the discretization error may not affect stability but it can induce some unexpected behaviour.
As we can see in Fig.~\ref{fig:sim1}, curves (ii) and (ie), the trajectories are crossing the sliding manifold.
This phenomenon can be explained by the following fact: let $\Delta_k$ be the discretization error on $u^{eq}$
at time
$t_k$. We have the recurrence equation $\bar{\sigma}_{k+1}=\bar{\sigma}_k + \Delta_k + CB^{*}\bar{u}^s_k$.
Let us consider the implicit discretization of $u^s$.
If $0<\bar{\sigma}_k<CB^{*}$, then the system should enter the discrete-time sliding phase.
However if $\Delta_k + \bar{\sigma}_k<-2CB^{*}$, then for any value of $\bar{u}^s_k$, $\bar{\sigma}_{k+1}<-CB^{*}$.
Hence, due to the discretization error, $\bar{u}^s$ fails to bring $\bar{\sigma}_{k+1}$ to $0$
and the trajectory of the system crosses the sliding manifold.
The same happens with the explicit discretization of $u^s$.
With the midpoint method in \eqref{UeqZOHm}, curves (mi) and (me), and with the new control scheme~\eqref{eq:new_control},
curve (ex), the system state reaches the sliding manifold.
\begin{figure}[ht]
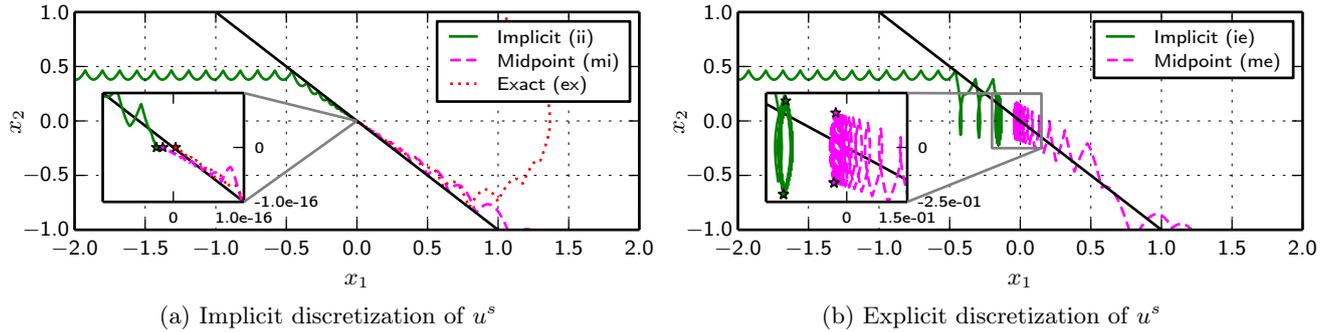

 \centering
 \begin{subfigure}[b]{\lenSubFig\linewidth}
  \centering
  \beginpgfgraphicnamed{2D-implicit-h-0.3_zoom}\input{pgf/2D-implicit-h-0.3_zoom.pgf}\endpgfgraphicnamed
  \caption{Implicit discretization of $u^s$}
  \label{fig:sim1_zoom:i}
 \end{subfigure}
 \begin{subfigure}[b]{\lenSubFig\linewidth}
  \centering
  \beginpgfgraphicnamed{2D-explicit-h-0.3_zoom}\input{pgf/2D-explicit-h-0.3_zoom.pgf}\endpgfgraphicnamed
  \caption{Explicit discretization of $u^s$}
  \label{fig:sim1_zoom:e}
 \end{subfigure}
 \caption{Detail of Fig.~\ref{fig:sim1}, $h=0.3$ s, $\alpha=1$.}
 \label{fig:sim1_zoom}
\end{figure}

Near the sliding manifold (Fig.~\ref{fig:sim1_zoom:i} and \ref{fig:sim1_zoom:e}),
the behaviour of the system is more sensitive to the discretization of $u^s$.
In the implicit case (method~\eqref{eq:Usi}, Fig.~\ref{fig:sim1_zoom:i}), in the discrete-time
sliding phase, $\bar{\sigma}_k$ is very close to $0$ ($\bar{\sigma}_k = 0$ with the exact method). In each case,
it converges to the origin (at the machine precision).
This is visible on the zoom box in Fig.~\ref{fig:sim1_zoom:i}, where markers indicate the state of the system 
at each time instant $t_k$, during the last second of each simulation.
When the explicit method~\eqref{eq:Use} is used, the system chatters around the sliding manifold, within
a neighborhood of order $h$ ($0.3$~s here), see Fig.~\ref{fig:sim1_zoom:e}.
\begin{figure}[ht]
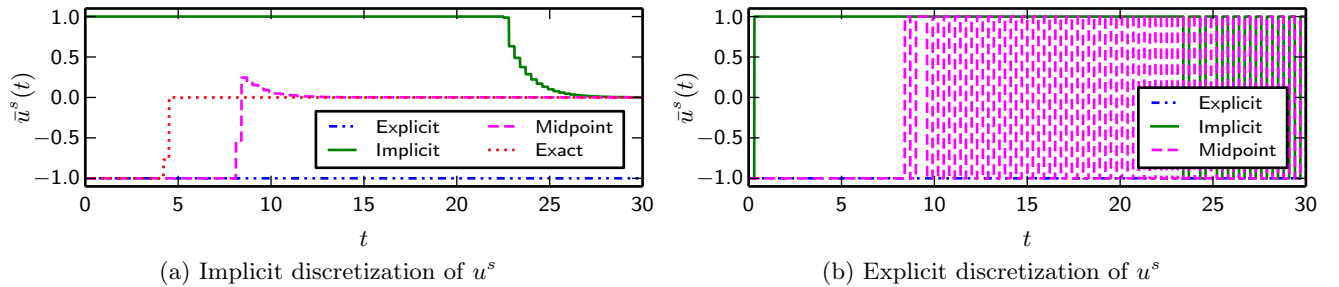

 \centering
 \begin{subfigure}[b]{\lenSubFig\linewidth}
  \centering
  \beginpgfgraphicnamed{2D-implicit-h-0.3_us}\input{pgf/2D-implicit-h-0.3_us.pgf}\endpgfgraphicnamed
  \caption{Implicit discretization of $u^s$}
  \label{fig:sim1_us:i}
 \end{subfigure}
 \begin{subfigure}[b]{\lenSubFig\linewidth}
  \centering
  \beginpgfgraphicnamed{2D-explicit-h-0.3_us}\input{pgf/2D-explicit-h-0.3_us.pgf}\endpgfgraphicnamed
  \caption{Explicit discretization of $u^s$}
  \label{fig:sim1_us:e}
 \end{subfigure}
 \caption{Evolution of $\bar{u}^s$ for different discretization methods, with $h=0.3$ s and $\alpha=1$.}
 \label{fig:sim1_us}
\end{figure}

In Fig.~\ref{fig:sim1_us:e}, the explicitly discretized discontinuous control
$\bar{u}^s$ takes its values in $\{-1, 1\}$ and starts at some point a limit cycle, as studied
in~\cite{galias2008analysis}. This cycle is also visible on the zoom box in Fig.~\ref{fig:sim1_zoom:e} with the help of
the markers. In Fig.~\ref{fig:sim1_us:i},
for each discretization of $u^{eq}$, $\bar{u}^s$ converges to $0$, which is the value that $u^s$ takes in the sliding phase.
In the implicit and midpoint cases, at the beginning of the discrete-time sliding phase, $\bar{u}^s$ takes
non zero values since there are discretization errors on $u^{eq}$. That is, if $\bar{\sigma}_k = 0$,
$\bar{\sigma}_{k+1} \neq 0$. The discontinuous control tries to bring $\bar{\sigma}_{k+1}$ to $0$ and
counteracts the error. As the state goes to the origin, the error converges to $0$.
These simulations illustrate the fact that error is smaller in the midpoint case than in the implicit case,
as shown in Lemma~\ref{lem:midpoint}. With the exact method of Section~\ref{sec:design},
$\bar{u}^s$ goes to $0$ after 1 timestep in the discrete-time sliding phase. In Fig.~\ref{fig:sim1_us:i}
and \ref{fig:sim1_us:e}, with the explicit discretization of $u^{eq}$, $\bar{u}^s$ takes always the same
value, since the closed-loop system moves away from the sliding manifold.
In terms of convergence to the sliding manifold, the first closed-loop system to enter the discrete-time sliding phase
is the exact method (Fig.~\ref{fig:sim1_us:i}), then the midpoint, finally the implicit method.
With the explicit method on $u^{eq}$, the system moves away from the sliding manifold and thus cannot enter the
discrete-time sliding phase.
\begin{figure}[ht]
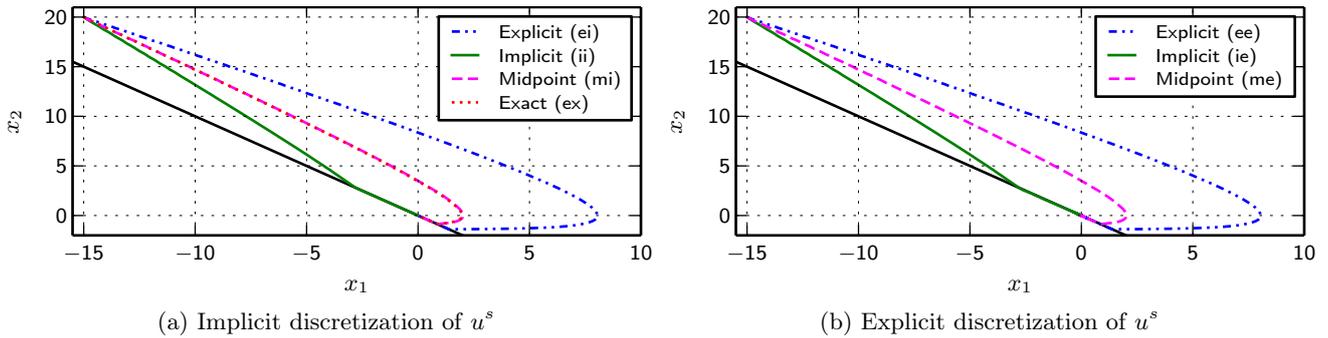

 \centering
 \begin{subfigure}[b]{\lenSubFig\linewidth}
  \centering
  \beginpgfgraphicnamed{2D-implicit-h-0.03}\input{pgf/2D-implicit-h-0.03.pgf}\endpgfgraphicnamed
  \caption{Implicit discretization of $u^s$}
  \label{fig:sim2:i}
 \end{subfigure}
 \begin{subfigure}[b]{\lenSubFig\linewidth}
  \centering
  \beginpgfgraphicnamed{2D-explicit-h-0.03}\input{pgf/2D-explicit-h-0.03.pgf}\endpgfgraphicnamed
  \caption{Explicit discretization of $u^s$}
  \label{fig:sim2:e}
 \end{subfigure}
 \caption{Simulations of system~\eqref{sys:2Dunstable} using different discretization methods, with $h = 0.03$ s and $\alpha=1$.}
 \label{fig:sim2}
\end{figure}
\begin{figure}[ht]
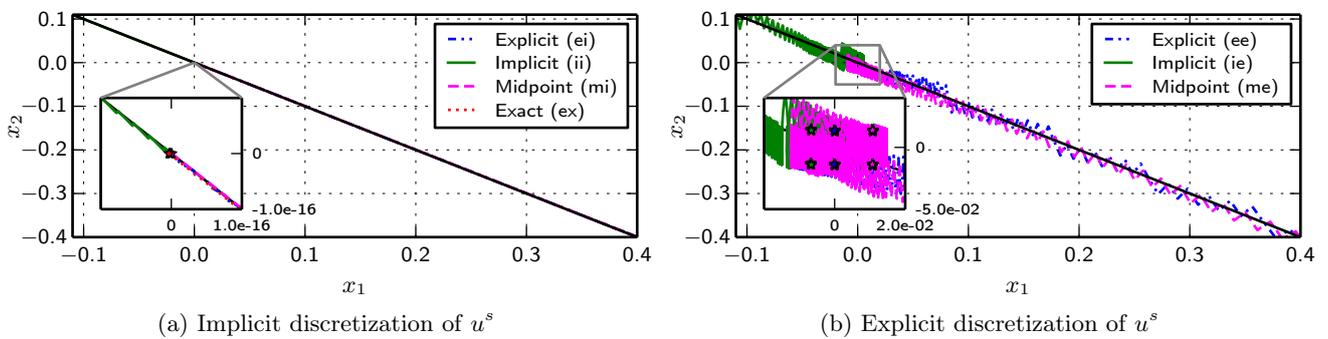

 \centering
 \begin{subfigure}[b]{\lenSubFig\linewidth}
  \centering
  \beginpgfgraphicnamed{2D-implicit-h-0.03_zoom}\input{pgf/2D-implicit-h-0.03_zoom.pgf}\endpgfgraphicnamed
  \caption{Implicit discretization of $u^s$}
  \label{fig:sim2_zoom:i}
 \end{subfigure}
 \begin{subfigure}[b]{\lenSubFig\linewidth}
  \centering
  \beginpgfgraphicnamed{2D-explicit-h-0.03_zoom}\input{pgf/2D-explicit-h-0.03_zoom.pgf}\endpgfgraphicnamed
  \caption{Explicit discretization of $u^s$}
  \label{fig:sim2_zoom:e}
 \end{subfigure}
 \caption{Detail of Fig.~\ref{fig:sim2}, $h=0.03$ s, $\alpha=1$.}
 \label{fig:sim2_zoom}
\end{figure}

The next set of simulations uses the same parameters as the previous one,
except for the timestep which is smaller: $h=0.03$~s.
In contrast with the results presented in Fig.~\ref{fig:sim1}, the closed-loop system is stable
in all cases, see Fig.~\ref{fig:sim2}.
As expected, the discretization error is smaller and no trajectory crosses the sliding manifold. 
It is not possible to distinguish the solutions associated with the midpoint from the one obtained with the exact
method in Fig.~\ref{fig:sim2:i}. In Fig.~\ref{fig:sim2_zoom:i} with the implicit discretization of $u^s$,
the states converge again to a very small ball near the origin.
In the explicit case, there is some numerical chattering, again with the same order of magnitude as the timestep
($h=0.03$~s, Fig.~\ref{fig:sim2_zoom:e}).
\begin{figure}[ht]
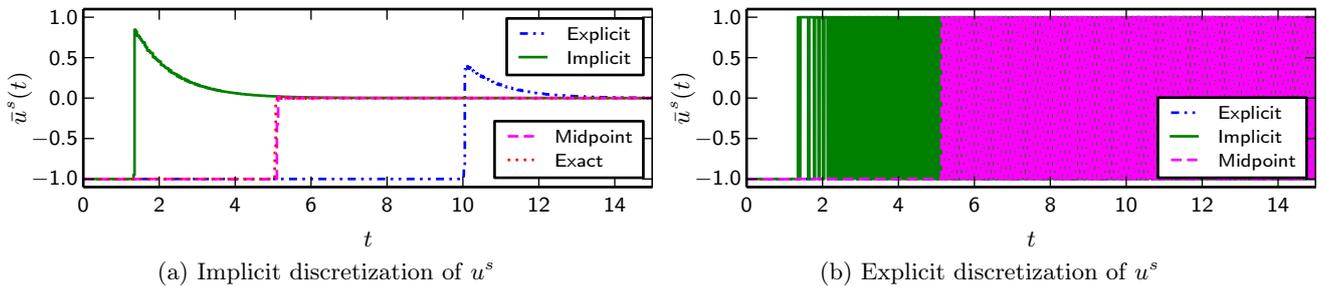

 \centering
 \begin{subfigure}[b]{\lenSubFig\linewidth}
  \centering
  \beginpgfgraphicnamed{2D-implicit-h-0.03_us}\input{pgf/2D-implicit-h-0.03_us.pgf}\endpgfgraphicnamed
  \caption{Implicit discretization of $u^s$}
  \label{fig:sim2_us:i}
 \end{subfigure}
 \begin{subfigure}[b]{\lenSubFig\linewidth}
  \centering
  \beginpgfgraphicnamed{2D-explicit-h-0.03_us}\input{pgf/2D-explicit-h-0.03_us.pgf}\endpgfgraphicnamed
  \caption{Explicit discretization of $u^s$}
  \label{fig:sim2_us:e}
 \end{subfigure}
 \caption{Evolution of $\bar{u}^s$ for different discretization methods, with $h = 0.03$ s and $\alpha=1$.}
 \label{fig:sim2_us}
\end{figure}
%
In Fig.~\ref{fig:sim2_us:i}, once in the discrete-time sliding phase, $\bar{u}^s$ counteracts the discretization error on $u^{eq}$,
which is smaller than in Fig.~\ref{fig:sim1_us:i}.
The discretization error for the midpoint discretization in \eqref{UeqZOHm} is much smaller, and
its curve overlaps completely with the one of the exact discretization method.

The results presented here bring into view the numerical chattering caused by an explicit discretization of $u^s$,
while the implicit method is free of it.
The importance of the discretization of $u^{eq}$ is also illustrated, with the explicit method leading to a diverging
system and the counterintuitive behaviour yielded by the implicit method. The exact method from Section~\ref{sec:design} produces
good results and in agreement with the theoretical results.
\subsection{Perturbed case}\label{sec:pert}

We now add a perturbation $\xi(t)$ in the system~\eqref{sys:2Dunstable}. In the next set of simulations, the perturbation is
$\xi(t) = 0.6\mathrm{ exp}(\min(6-t,0))\sin(2\pi t)$.
Note that for all $t$, $\|\xi(t)\|\leq0.6$. This particular $\xi$ has been chosen to highlight
that if the perturbation vanishes, with the implicit discretization in~\eqref{eq:Usi},
$\bar{u}^s$ goes to $0$, whereas in the explicit case~\eqref{eq:Use},
$\bar{u}^s$ continues to switch between $-1$ and $1$.
%
%
\begin{figure}[ht]
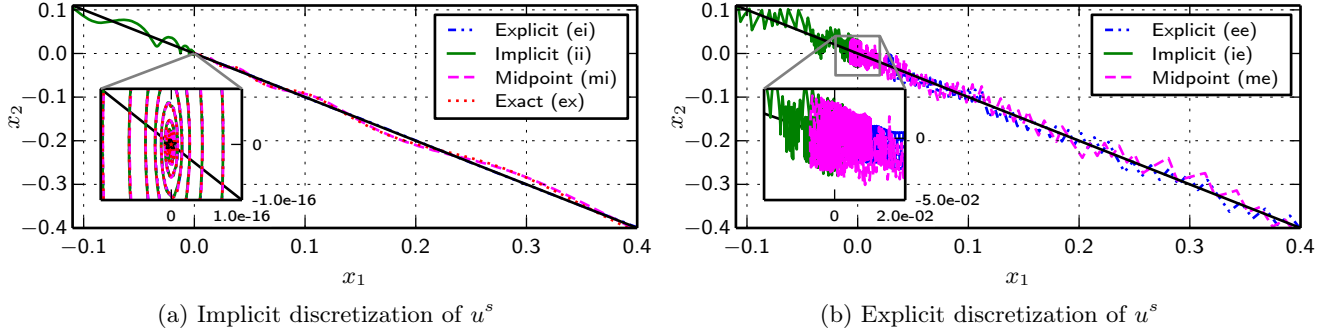

 \centering
 \begin{subfigure}[b]{\lenSubFig\linewidth}
   \centering
   \beginpgfgraphicnamed{2D-pert-implicit-h-0.03_zoom}\input{pgf/2D-pert-implicit-h-0.03_zoom.pgf}\endpgfgraphicnamed
   \caption{Implicit discretization of $u^s$}
   \label{fig:sim3_zoom:i}
 \end{subfigure}
 \begin{subfigure}[b]{\lenSubFig\linewidth}
  \centering
  \beginpgfgraphicnamed{2D-pert-explicit-h-0.03_zoom}\input{pgf/2D-pert-explicit-h-0.03_zoom.pgf}\endpgfgraphicnamed
  \caption{Explicit discretization of $u^s$}
  \label{fig:sim3_zoom:e}
 \end{subfigure}
 \caption{Simulations of system~\eqref{sys:2Dunstable} using different
 discretization methods for $u^{eq}$ and $h=0.03$~s (perturbed case).}
 \label{fig:sim3_zoom}
\end{figure}
With the implicit discretization of $u^s$ (Fig.~\ref{fig:sim3_zoom:i}) the closed-loop system
enters the discrete-time sliding phase at some point. Then $\bar{u}_k^s=-(CB^{*})^{-1}Cp_{k-1}$,
if some simple assumptions are satisfied, see the proof of Proposition~\ref{prop:pert} below.
It takes such value in order to counteract the effect of the perturbation during the elapsed time interval,
hence imitating the continuous-time Filippov solutions.
However the trajectories are now clearly only in a neighborhood of the sliding manifold.
Finally in each case in Fig.~\ref{fig:sim3_us:i}, $\bar{u}^s_k$ settles to $0$, as in continuous time. Indeed,
the perturbation $\xi$ used in this simulation goes to 0 exponentially fast at some point.
On the other hand, with an explicit discretization of $u^s$ (Fig.~\ref{fig:sim3_us:e}), it is much harder to witness the
influence of the perturbation on $\bar{u}^s$ since filtering would be necessary to see the effect.
The control input chattering is striking with the explicit discretization in Fig.~\ref{fig:sim1_us:e}, \ref{fig:sim2_us:e},
and \ref{fig:sim3_us:e}.
\begin{figure}[ht]
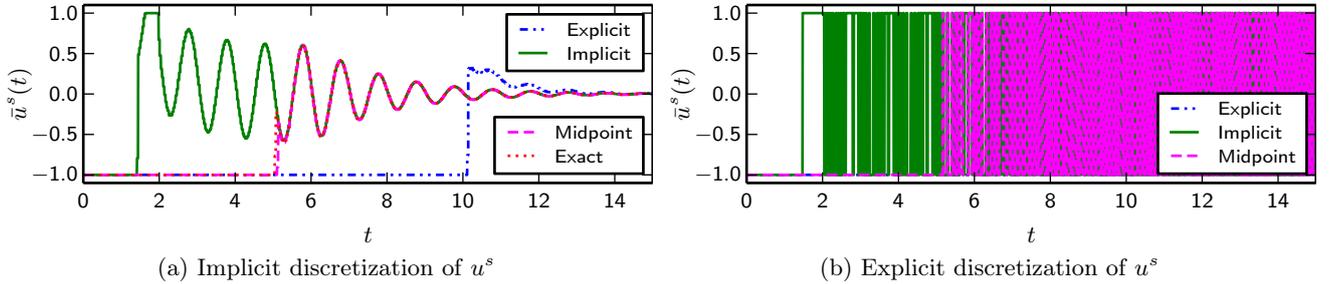

 \centering
 \begin{subfigure}[b]{\lenSubFig\linewidth}
  \centering
  \beginpgfgraphicnamed{2D-pert-implicit-h-0.03_us}\input{pgf/2D-pert-implicit-h-0.03_us.pgf}\endpgfgraphicnamed
  \caption{Implicit discretization of $u^s$}
  \label{fig:sim3_us:i}
 \end{subfigure}
 \begin{subfigure}[b]{\lenSubFig\linewidth}
  \centering
  \beginpgfgraphicnamed{2D-pert-explicit-h-0.03_us}\input{pgf/2D-pert-explicit-h-0.03_us.pgf}\endpgfgraphicnamed
  \caption{Explicit discretization of $u^s$}
  \label{fig:sim3_us:e}
 \end{subfigure}
 \caption{Evolution of $\bar{u}^s$ for different discretization methods for $u^{eq}$ and $u^s$, $h=0.03$~s. (perturbed case)}
 \label{fig:sim3_us}
\end{figure}
\begin{figure}[ht]
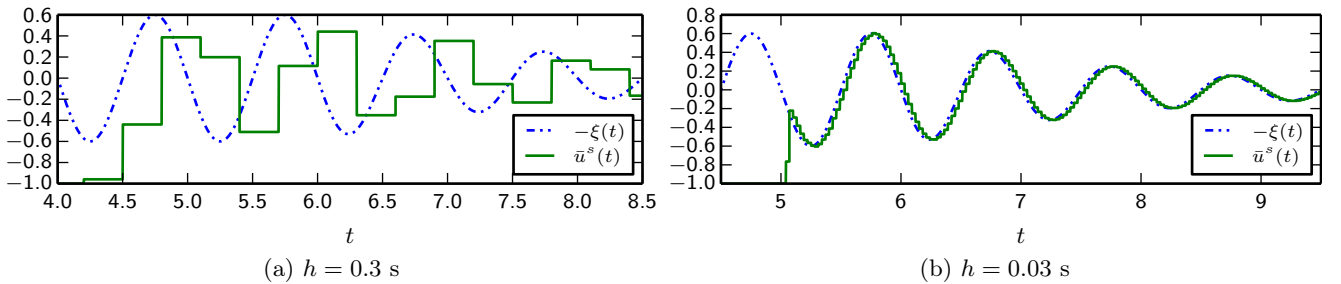

 \centering
 \begin{subfigure}[b]{\lenSubFig\linewidth}
  \centering
  \beginpgfgraphicnamed{2D-pert-correspondence-h-0.3_us}\input{pgf/2D-pert-correspondence-h-0.3_us.pgf}\endpgfgraphicnamed
  \caption{$h = 0.3$ s}
  \label{fig:sim3_corr1}
 \end{subfigure}
 \begin{subfigure}[b]{\lenSubFig\linewidth}
  \centering
  \beginpgfgraphicnamed{2D-pert-correspondence-h-0.03_us}\input{pgf/2D-pert-correspondence-h-0.03_us.pgf}\endpgfgraphicnamed
  \caption{$h = 0.03$ s}
  \label{fig:sim3_corr2}
 \end{subfigure}
 \caption{Evolution of $\bar{u}^s$ and the perturbation using the new control scheme for two different timesteps.}
 \label{fig:sim3_corr}
\end{figure}
In Fig.~\ref{fig:sim3_corr} we further illustrate the phenomenon in the implicit case:
$\bar{u}^s$ approximates $-\xi$ with a delay depending on $h$. This is close to the behaviour one expects from the
Filippov's framework in continuous-time. We shall investigate this phenomenon in Section~\ref{sec:stab_analysis}.

The following simulation results illustrate that with an implicit discretization of $u^{s}$,
the chattering in the discrete-time sliding phase is solely due to the perturbation.
The setup is the same as in Section~\ref{sec:simNom},
except that there is a perturbation $\xi(t) = 0.9\sin(t)$ and $\alpha$, the magnitude of the discontinuous control,
changes. For the present set of simulation, we use $\alpha=1, 3$ and $10$, values large enough to ensure that the
perturbation is always dominated by the control. On Fig.~\ref{fig:effect-alpha:i}, we cannot distinguish the three
trajectories in the discrete-time sliding phase, since even if $\bar{u}^s_k$ takes value in a larger set, the selected
value within $(\alpha,\alpha)$ does not change.
This is supported by the Fig.~\ref{fig:effect-alpha_us:i}: it is again not possible to differentiate
the values taken by the controllers.
On the contrary, in Fig.~\ref{fig:effect-alpha:e}, the three different trajectories are clearly
visible. Here we observe that the numerical chattering is dominant. Each time we increase $\alpha$, the amplitudes
of the oscillation around the manifold are getting bigger. We also observe that the precision of the system seems to
be affected by the magnitude of the control. The bigger it is, the farther the system oscillates from the origin.
\begin{figure}[ht]
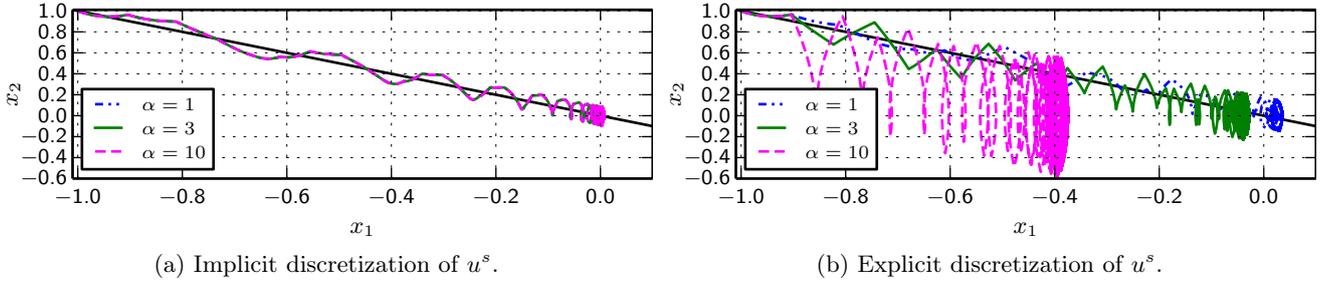

 \centering
 \begin{subfigure}[b]{\lenSubFig\linewidth}
   \centering
   \beginpgfgraphicnamed{2D-implicit-effect-alpha}\input{pgf/2D-implicit-effect-alpha.pgf}\endpgfgraphicnamed
   \caption{Implicit discretization of $u^s$.}
   \label{fig:effect-alpha:i}
 \end{subfigure}
 \begin{subfigure}[b]{\lenSubFig\linewidth}
  \centering
  \beginpgfgraphicnamed{2D-explicit-effect-alpha}\input{pgf/2D-explicit-effect-alpha.pgf}\endpgfgraphicnamed
  \caption{Explicit discretization of $u^s$.}
  \label{fig:effect-alpha:e}
 \end{subfigure}
 \caption{Simulations of system~\eqref{sys:2Dunstable} with a perturbation, using different values for $\alpha$
 and with $h = 0.1\second$.}
 \label{fig:effect-alpha}
\end{figure}
\begin{figure}[ht]
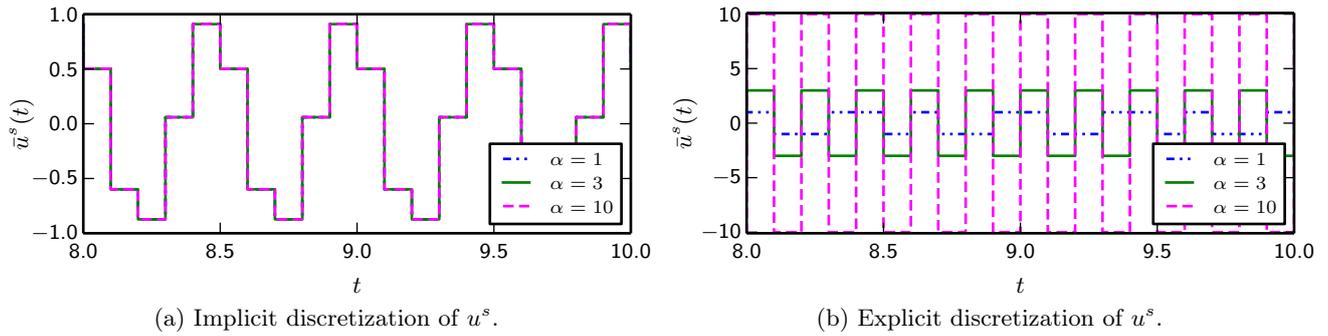

 \centering
 \begin{subfigure}[b]{\lenSubFig\linewidth}
   \centering
   \beginpgfgraphicnamed{2D-implicit-effect-alpha_us}\input{pgf/2D-implicit-effect-alpha_us.pgf}\endpgfgraphicnamed
   \caption{Implicit discretization of $u^s$.}
   \label{fig:effect-alpha_us:i}
 \end{subfigure}
 \begin{subfigure}[b]{\lenSubFig\linewidth}
  \centering
  \beginpgfgraphicnamed{2D-explicit-effect-alpha_us}\input{pgf/2D-explicit-effect-alpha_us.pgf}\endpgfgraphicnamed
  \caption{Explicit discretization of $u^s$.}
  \label{fig:effect-alpha_us:e}
 \end{subfigure}
 \caption{Evolution of $\bar{u}^s$, using different values for $\alpha$ and with $h = 0.1$~s.}
 \label{fig:effect-alpha_us}
\end{figure}

\section{Stability properties}\label{sec:stab_analysis}

\subsection{Nominal case}

In this subsection, the stability of the system~\eqref{DTMMsys} is analyzed. Using the equivalent control
proposed in Section~\ref{sec:design}, the obtained properties can be transposed to the original nominal
system~\eqref{linSyst}.
Note that the mapping $\Sgn(\cdot)$, as introduced in Definition~\ref{de:Sgn}, has the following properties:
\begin{align}
  \langle v_1 - v_2, x_1 - x_2\rangle &\geq 0, \;
  \forall v_{i}\in\Sgn(x_i), i=1, 2\label{MMprop1}\\
  0&\notin \Sgn(x),\;\forall x\neq0.\label{MMprop2}
\end{align}
Property \eqref{MMprop1} is known as the \emph{monotonicity} of the $\Sgn$ set-valued function.
The positive-definitiveness property of $CB^*$ is pivotal to the results presented in this section.
Even if it is not explicit with the current notations, $CB^*$ depends on the timestep $h$. The following
lemma gives some insight of when this condition is fulfilled.
\begin{lem}\label{lem:pos_def}
 Suppose that $CB$ is positive-definite. There exists an interval $I=[0, h^*]\subset\R_+$,
 $h^*>0$ such that if the timestep $h$ is chosen in $I$, then $CB^*/h$ is positive-definite (and so is $CB^*$ if $h\neq0$).
\end{lem}
\begin{proof}
 Let $h>0$, $CB_s$ and $CB^*_s$ are the symmetric part of $CB$ and $CB^*$, respectively.
 Let $\Delta\coloneqq CB^*_s/h - CB_s =
 \sum_{l=1}^{\infty}\frac{CA^lB+B^T(A^l)^TC^T}{2(l+1)!}h^l = \mathcal{O}(h)$.
 Since $CB_s$ is symmetric, it is also normal. Hence we can apply Corollary~4.2.16, p.~405 in~\cite{hinrichsen2005mathematical},
 which yields that for any eigenvalue $\mu$ of $CB^*_s/h$, $\min_\lambda |\lambda-\mu|\leq\|\Delta\|_{2,2}$, with
 $\lambda$ an eigenvalue of $CB_s$ and $\|\cdot\|_{2,2}$ the spectral norm. By definition, $\Delta$ is a symmetric
 matrix with real entries. Hence $\|\Delta\|_{2,2}=\delta_{max}$, the largest module of any eigenvalue of $\Delta$.
 Let $\gamma>0$ be the smallest eigenvalue of $CB_s$. If $\delta_{max}<\gamma$, then every eigenvalue of $CB^*_s/h$ is positive
 and since $CB^*_s/h$ is by definition symmetric, $CB^*_s/h$ is positive definite.
 It is easy to see that $\Delta\to0$ as $h\to0$ and that $\Delta$ depends continuously on $h$. Therefore by
 Corollary~4.2.4, p.~399 in~\cite{hinrichsen2005mathematical}, the eigenvalues of $\Delta$ are continuous functions of
 $h$. Then it is always possible to find $h^*$ such that $\delta_{max}<\gamma$ for all $h<h^*$, which implies that
 $CB^*_s/h$ is positive-definite and finally $CB^*_s$ is positive definite.
\end{proof}
\begin{lem}\label{lem:cv1}
 If $CB^{*}$ is symmetric positive-definite, then the equilibrium state $\bar{\sigma}^* = 0$
 of \eqref{DTMMsys} is globally Lyapunov stable.
\end{lem}
\begin{proof}
 Let $V(\bar{\sigma}_k) \coloneqq \bar{\sigma}_k^TP\bar{\sigma}_k$ with $P = \left(CB^{*}\right)^{-1}$,
 be a candidate Lyapunov function.
 Along the trajectories of the system~\eqref{DTMMsys}, one obtains:
 \begin{align}
  V(\bar{\sigma}_{k+1}) - V(\bar{\sigma}_k) &= \bar{\sigma}_{k+1}^TP\bar{\sigma}_{k+1} - \bar{\sigma}_k^TP\bar{\sigma}_k\notag\\
  &= \bar{\sigma}_{k+1}^TP\bar{\sigma}_{k+1} - (\bar{\sigma}_{k+1}
  - CB^{*}\bar{u}^s_{k})^TP(\bar{\sigma}_{k+1} - CB^{*}\bar{u}^s_{k})\notag\\
  &= -(\bar{u}^s_{k})^TCB^{*}\bar{u}^s_{k} + 2(\bar{u}^s_{k})^T\bar{\sigma}_{k+1}.\label{eq:lem:cv1}
 \end{align}
 Using \eqref{MMprop1} with $v_1 = \bar{u}^s_{k}$, $v_2 = 0$, $x_1 = \bar{\sigma}_{k+1}$, and $x_2 = 0$ yields
 $(\bar{u}^s_{k})^T\bar{\sigma}_{k+1}\leq0$. Since $CB^{*}$ is positive-definite,
 the first term is always nonpositive. This completes the proof.
\end{proof}
We can also use a non-quadratic Lyapunov function, inspired by the one presented in \cite{utkin1992sliding}.
As we shall see, it relaxes the symmetry condition on the matrix $CB^{*}$.
\begin{lem}\label{lem:cv2}
 If $CB^{*}$ is positive-definite, then the equilibrium state $\bar{\sigma}^* = 0$ of~\eqref{DTMMsys}
 is globally Lyapunov stable.
 \end{lem}
\begin{proof}
 Let $V(\bar{\sigma}_k) \coloneqq -(\bar{u}^s_{k-1})^T\bar{\sigma}_k$ be the candidate Lyapunov function,
 and $\bar{u}^s_{k-1}\in -\alpha\Sgn(\bar{\sigma}_k)$.
 The function $V$ is positive definite, radially unbounded, and decrescent since
 $-(\bar{u}^s_{k-1})^T\bar{\sigma}_k = \alpha\|\bar{\sigma}_k\|_1^2$ and $\alpha>0$.
 Let us study the variations of $V$:
 \begin{align}
  V(\bar{\sigma}_{k+1}) - V(\bar{\sigma}_k) &= -(\bar{u}^s_{k})^T\bar{\sigma}_{k+1} + (\bar{u}^s_{k-1})^T\bar{\sigma}_k\notag\\
  &= -(\bar{u}^s_{k})^T(\bar{\sigma}_k+CB^{*}\bar{u}^s_{k})+(\bar{u}^s_{k-1})^T\bar{\sigma}_k\notag\\
  &= -(\bar{u}^s_{k})^TCB^{*}\bar{u}^s_{k} + \langle\bar{u}^s_{k-1}-\bar{u}^s_{k}, \bar{\sigma}_k\rangle.
  \label{eq:lem:cv2}
 \end{align}
 The first term is always nonpositive with the hypothesis on $CB^{*}$. For the second term,
 if $\bar{u}^s_{k}\in-\alpha\Sgn(\bar{\sigma}_{k+1})$,
 then it also belongs to $-\alpha\Sgn(0)$. Hence, by \eqref{MMprop1}, the second term is always nonpositive.
This completes the proof.
\end{proof}
\begin{prop}\label{prop:ft}
 If the hypothesis of either Lemma~\ref{lem:cv1} or \ref{lem:cv2} are satisfied, then
 the fixed point $(\bar{\sigma}, \bar{u}) = (0,0)$ of \eqref{DTMMsys} is globally finite-time Lyapunov stable.
\end{prop}
\begin{proof}
 In each case, the difference $V(\bar{\sigma}_{k+1}) - V(\bar{\sigma}_k)$ consists of $-(\bar{u}^s_{k})^TCB^{*}\bar{u}^s_{k}$
 plus a nonpositive term. Since $CB^{*}$ is positive-definite, it holds that
 $-(\bar{u}^s_{k})^TCB^{*}\bar{u}^s_{k}\leq-\beta\|\bar{u}^s_{k}\|^2$, with $\beta>0$ the smallest eigenvalue of $CB_s^{*}$. Note that if $\bar{\sigma}_{k+1}\neq0$,
 then $\|\bar{u}^s_{k}\| \geq \alpha$ and $V(\bar{\sigma}_{k+1}) - V(\bar{\sigma}_k) \leq -\alpha^2\beta$. Iterating, one obtains
 $V(\bar{\sigma}_{k+1}) - V(\bar{\sigma}_0) \leq -k\alpha^2\beta$. Let $k_0\coloneqq \lceil V(\bar{\sigma}_0)/\beta\alpha^2\rceil$.
 Suppose $V(\bar{\sigma}_{k_0+1})\neq0$.
 Then $V(\bar{\sigma}_{k_0+1})-V(\bar{\sigma}_0) \leq -k_0\alpha^2\beta \leq -V(\bar{\sigma}_0)$.
 This yields $V(\bar{\sigma}_{k_0+1})\leq0$, which implies $V(\bar{\sigma}_{k_0+1})=0$.
 Then $\bar{\sigma}_{k_0+1} = 0$ and $\bar{\sigma}_{k}=0$ for all $k>k_0$.
\end{proof}

\subsection{Perturbed case}\label{sec:stab_analysis:pert}

Let us now consider the case with perturbation. The evolution of the sliding variable $\bar{\sigma}$ is
governed by 
\begin{gather}
 \bar{\sigma}_{k+1} = \bar{\sigma}_k + CB^{*}\bar{u}^s_k + Cp_k,\label{eq:dt_sys_pert}
\end{gather}
where $p_k\coloneqq \int_{t_k}^{t_{k+1}}\!e^{A(t_{k+1}-\tau)}B\xi(\tau)\mathrm{d}\tau$ and $\bar{u}^s_k$ is the unique solution of the generalized equation~\eqref{lcp:inexact}.
Although the system will never reach and stay on the sliding manifold as in the continuous-time case,
it enters the discrete-time sliding phase as stated in Definition~\ref{de:dtSlidingPhase}
 and stays in it. Let us first present a technical lemma.
\begin{lem}\label{lem:tech_matrix_estimation}
 Let $M\in\R^{n\times n}$ and $M_s\coloneqq 1/2(M+M^T)$. Suppose $M$ is positive-definite.
 Let $\beta>0$ be the smallest eigenvalue of $M_s$. Then for all $x\in\R^n$, $\|M^{-1}x\|\leq\beta^{-1}\|x\|$.
\end{lem}
\begin{proof}
 $M^{-1}$ exists since $M_s$ is positive-definite. Let $\nu_{min}$ (resp. $\nu_{max}$) be the smallest (resp. largest) singular
 values of $M$ (resp. $M^{-1}$). Two relations hold: $\nu_{max}=\nu_{min}^{-1}$ \cite[Fact~6.3.21, p.233]{bernstein2005matrix} and
 $\nu_{min}\geq\beta>0$ \cite[Corollary~3.1.5,~p.151]{horn1991topics}. Then using the spectral norm definition,
 $\|M^{-1}x\|\leq\|M^{-1}\|\|x\|=\nu_{max}\|x\|\leq\beta^{-1}\|x\|$.
\end{proof}
From now on, let $CB^{*}_s \coloneqq 1/2(CB^{*}+(CB^{*})^T)$ and let $\beta$ be its smallest eigenvalue.
\begin{prop}\label{prop:pert}
 Suppose that $CB^{*}$ is positive-definite.
 If $\alpha>0$ is such that for all $k\in\N$ $\|Cp_k\|<\alpha\beta$,
 then the perturbed closed-loop system~\eqref{eq:dt_sys_pert}-\eqref{lcp:inexact} will enter
 the discrete-time sliding phase in finite time and stay in it.
 Furthermore if $h\in[0,h^*]$, as defined in Lemma~\ref{lem:pos_def}, then there exists an upper bound $T^*$ on
 the duration of the reaching phase.
\end{prop}
\begin{proof}
 Let $V(\bar{\sigma}_k) \coloneqq -\bar{u}^{s T}_{k-1}\bar{\sigma}_k$, $\bar{u}^s_{k-1}\in -\alpha\Sgn(\bar{\sigma}_k)$.
 Assume that the system is initialized outside the discrete-time sliding phase.
 From Definition~\ref{de:dtSlidingPhase}, it follows that $\|\bar{u}^s_{k}\|\geq\alpha$.
 Starting from~\eqref{eq:lem:cv2}, doing as in the proof of Proposition~\ref{prop:ft}
 and adding the contribution of the perturbation,
 we have $V(\bar{\sigma}_{k+1}) - V(\bar{\sigma}_k) \leq -\beta\|\bar{u}^s_{k}\|^2-(\bar{u}^s_{k})^TCp_k$.
 Using the Cauchy-Schwarz inequality, we obtain $|(\bar{u}^s_{k})^TCp_k|\leq\|\bar{u}^s_{k}\|\|Cp_k\|$.
 To ensure that $V$ decreases strictly, we need $\|Cp_k\|<\beta\|\bar{u}^s_{k}\|$. This condition is satisfied using the
 hypothesis on the gain $\alpha$ and the fact that $\beta>0$.
 Note that even in the case with multiple switching surfaces, $V$ decreases as long
 as the system is not ``sliding'' on the intersection of all the manifolds.
 If $\widetilde{\sigma}_{k+1}=0$, then we enter the discrete-time sliding phase.
 The finite-time property is derived as in the proof of Proposition~\ref{prop:ft}. Let $\kappa = \alpha\beta-\|Cp_k\|$.
 From the assumption, $\kappa>0$ holds.
 In the reaching phase, $V$ decreases by at least $\kappa\alpha$ at each timestep.
 Hence, $V(\bar{\sigma}_k)$ converges to $0$ in finite-time.
 Now if the system is in the discrete-time sliding phase at $t_k$,
 then $\widetilde{\sigma}_{k+1}=0$ and $\bar{\sigma}_{k+1} = Cp_k$.
 At time $t_{k+1}$, we have $\widetilde{\sigma}_{k+2} = Cp_k + CB^{*}\bar{u}^s_{k+1}$.
 Let us show that $\bar{u}^s_{k+1} = -(CB^{*})^{-1}Cp_k$
 is the unique solution to the generalized equation~\eqref{lcp:inexact}. With this value, $\widetilde{\sigma}_{k+2} = 0$. 
 Using Lemma~\ref{lem:tech_matrix_estimation} and the hypothesis, $\|\bar{u}^s_{k+1}\|\leq\beta^{-1}\|Cp_k\|<\alpha$.
 Relations between norms yield $\|\bar{u}^s_{k+1}\|_\infty<\alpha$.
 Then $\bar{u}^s_{k+1}\in(\alpha,\alpha)^p\subset\alpha\Sgn(0)$ and $\bar{u}^s_{k+1}$ is a solution to~\eqref{lcp:inexact}.
 With the hypothesis of the proposition, $CB^{*}$ is also a $\mathbf{P}$-matrix. Then Lemma~\ref{lem:uniqueness}
 can be applied and yields the uniqueness property.
 Thus $\bar{u}^s_{k+1} = -(CB^{*})^{-1}Cp_k$ is the unique solution to \eqref{lcp:inexact} at time $t_{k+1}$,
 and by induction, the system stays in the discrete-time sliding phase.

 In the following, we suppose that $0<h<h^*$. Let $\delta_{max}^* = \|\Delta\|_{2,2}$ when $h=h^*$.
 From the expression of $k_0$ in the proof of Proposition~\ref{prop:ft},
 the duration of the reaching phase is $hk_0 < \frac{V(\bar{\sigma}_0)h}{\beta\alpha^2}+h$.
 Note that $\beta/h$ is an eigenvalue of $CB^*_s/h$. Applying again Corollary~4.2.16,
 p.~405 in~\cite{hinrichsen2005mathematical}, we have $\lambda-\beta/h\leq\delta_{max}\leq\delta_{max}^*$,
 with $\lambda$ an eigenvalue of $CB_s$.
 This yields $\gamma-\delta_{max}^*\leq\lambda-\delta_{max}^*\leq\beta/h$.
 Using this in the previous inequality, we get $hk_0 < \frac{V(\bar{\sigma}_0)}{\alpha^2(\gamma-\delta_{max}^*)}+h^*\eqqcolon T^*$.
\end{proof}
\begin{remark}
 In continuous time, the condition commonly found is $\alpha>\|\xi\|_{\infty,\R_+}$.
 If the perturbation $\xi$ is continuous, then it is possible to link this condition
 to the one used in the previous theorem, $\|Cp_k\|<\alpha\beta$ for all $k\in\N$.
 Using the mean value theorem for integration, we get $Cp_k = hCe^{A(t_{k+1}-t^\prime)}B\xi(t^\prime)
 = hCB\xi(t^\prime) + \mathcal{O}(h^2)$, with $t^\prime\in[t_k,t_{k+1}]$. Hence the first-order estimation for $\|Cp_k\|$ is
 $h\|CB\|_{2,2}\|\xi\|_2\leq\|CB\|_{2,2}\sqrt{p}\|\xi\|_{\infty}$. From the proof of Lemma~\ref{lem:pos_def},
 we have $\beta=h\lambda_{min}(CB) + \mathcal{O}(h^2)$.
 Then we get $\alpha>\|CB\|_{2,2}\sqrt{p}\|\xi\|_{\infty}/\beta + \mathcal{O}(h)$.
 Note that $\|CB\|_{2,2}>\lambda_{max}(CB)$ \cite[Corollary~3.1.5,~p.151]{horn1991topics}.
 Therefore $\|CB\|_{2,2}\sqrt{p}\|\xi\|_{\infty}/\beta\geq1$ and for $h$ small enough,
 $\|Cp_k\|<\alpha\beta$ implies $\alpha>\|\xi\|_{\infty,\R_+}$. If the sliding variable
 is a scalar, then the converse is also true at the limit.
\end{remark}
%
In the classical literature on discrete-time sliding mode, dealing with the explicit discretization~\eqref{eq:Use}
\cite{milosavljevic1985general,sarpturk1987stability,furuta1990sliding}, two conditions related to the sliding variable
emerged: $(\bar{\sigma}_{k+1}-\bar{\sigma}_{k})_i(\sigma_k)_i<0$ for all $i=1,\dots,n$, which is necessary;
and the second one is $|(\bar{\sigma}_{k+1})_i|<|(\bar{\sigma}_{k})_i|$. The conditions for linear systems,
stated in Lemma~\ref{lem:cv1}, \ref{lem:cv2}, and Proposition~\ref{prop:pert}
are directly on the system parameters and not on the evolution on the sliding variable,
which derives from the dynamics. This fact is, in our sense, much closer to the stability results obtained
in continuous-time \cite{utkin1992sliding}.

Let us now turn our attention to the relationship between $u$ and $\bar{u}$. In particular, we study the convergence
of $\bar{u}$ to $u$ during the discrete-time sliding phase, which is established after $T^*<+\infty$.
\begin{prop}\label{lem:cv_us_infty}
 Consider the discrete-time closed-loop system given by~\eqref{eq:dt_sys_pert} and~\eqref{lcp:inexact}.
 Let $\{h_n\}_{n\in\N}$ be any strictly decreasing sequence of positive numbers
 converging to $0$ and with $h_0<h^*$ (see Lemma~\ref{lem:pos_def}).
 Suppose that the perturbation $\xi\colon\R\to\R^p$ is uniformly continuous, that $CB$ is positive-definite
 and that $\alpha>0$ is chosen such that the conditions of Proposition~\ref{prop:pert} are satisfied for
 each timestep $h_n$.
 Then for any $S\subseteq[T^*, \infty)$, $\lim_{h_n\to 0} \lVert\bar{u}^s-u^s\rVert_{\infty,S} = 0$.
\end{prop}
\begin{proof}
 Let $\{t_k\}$ be a sequence such that for all $k\in\N$, $t_{k+1}-t_k=h_n$. For the sake of clarity,
 we omit to write explicitly the dependence on $n$.
 From Proposition~\ref{prop:pert} and Lemma~\ref{lem:pos_def},
 we know that for all $k$ such that $t_k\geq T^*$, $\bar{u}^s_k=-(CB^{*})^{-1}Cp_{k-1} =
 -(CB^*)^{-1}C\int_{t_{k-1}}^{t_{k}}\!e^{A(t_{k}-\tau)}B\xi(\tau)\mathrm{d}\tau$. During the sliding phase,
 the continuous-time controller satisfies $u^s(t) = - \xi(t)$. Let $S$ be any time interval contained in $[T^*, +\infty)$.
 Let $t\in S$ and $k\in\N$ is such that $t\in[t_k, t_{k+1})$. Hence $\bar{u}^s(t) = \bar{u}^s_k$.
  Let us study $\bar{u}_k^s-u^s(t)$:
  \begin{align}
   \bar{u}_k^s-u^s(t) &= -(CB^*)^{-1}C\int_{t_{k-1}}^{t_{k}}\!e^{A(t_{k}-\tau)}B\xi(\tau)\mathrm{d}\tau +\xi(t)\\
   &= -(CB^*)^{-1}\left( \int_{t_{k-1}}^{t_{k}}\!Ce^{A(t_{k}-\tau)}B\xi(\tau)\mathrm{d}\tau - CB^*\xi(t) \right).
  \end{align}
   Using~\eqref{eq:Bstar}, we obtain:
   \begin{align}
   \bar{u}_k^s-u^s(t) &= -(CB^*)^{-1}\left( \int_{t_{k-1}}^{t_{k}}\!Ce^{A(t_{k}-\tau)}B\xi(\tau)\mathrm{d}\tau - \int_{t_{k}}^{t_{k+1}}\!Ce^{A(t_{k+1}-\tau)}B\xi(t)\mathrm{d}\tau \right).\label{eq:diff_int}\\
   \intertext{With the change of variable $\tau^\prime=\tau+h_n$ in the first integral,
    we can group the two integrals in \eqref{eq:diff_int} as follows:}
   \bar{u}_k^s-u^s(t) &= - (CB^*)^{-1}\left( \int_{t_{k}}^{t_{k+1}}\!Ce^{A(t_{k+1}-\tau)}B(\xi(\tau-h_n)-\xi(t)\mathrm{d}\tau\right)\\
   &= -(CB^*)^{-1}\left( \int_{t_{k}}^{t_{k+1}}\! CB(\xi(\tau-h_n)-\xi(t))
   + \sum_{l=1}^{\infty}\frac{CA^lB}{l!}\left((t_{k+1}-\tau)^l(\xi(\tau-h_n)-\xi(t))\right)\mathrm{d}\tau \right).\label{eq:cv_1}
  \end{align}
 Using again~\eqref{eq:Bstar}, let us provide an approximation for the first factor:
 \begin{align}
  (CB^*)^{-1} &= \left( h_nCB + \sum_{l=1}^\infty \frac{CA^lB}{(l+1)!}h_n^{l+1} \right)^{-1}\\
  &= \left(I+\sum_{l=1}^{\infty}\frac{(CB)^{-1}CA^lB}{(l+1)!}h_n^{l}\right)^{-1}(h_nCB)^{-1}\\
  &= \left( I - \frac{(CB)^{-1}CAB}{2}h_n + \mathcal{O}(h_n^2) \right)(h_nCB)^{-1}.\label{eq:dev_taylor_CB}
 \end{align}
 The Taylor expansion holds if $\sum_{l=1}^{\infty}\frac{(CB)^{-1}CA^lB}{(l+1)!}h_n^{l}$ has all its eigenvalues in the unit
 disk. This is a mere technical restriction, since it is always possible to find a small enough positive number
 $h_{n_0}$ such this condition is satisfied. Since we are interested in the case where $\{h_n\}$ converges to
 $0$, this requirement is supposed to hold.
 For the first term in the integrand in \eqref{eq:cv_1}, we apply the mean value theorem for integration.
 If $\xi$ is a vector-valued function (this is the case when the sliding variable has dimension greater than $1$),
 we apply the theorem for each component separately. This yields for $i=1,\dots,n\quad
  (\int_{t_{k}}^{t_{k+1}}\! \xi(\tau-h_n)-\xi(t)\mathrm{d}\tau)_i = h_n(\xi_i(t^\prime_i-h_n)-\xi_i(t))$ for some
  $t^\prime_i\in[t_k,t_{k+1}]$.
  For the second part of the integrand in~\eqref{eq:cv_1}, we exchange the summation and integral signs. This is
  possible since the matrix exponential converges normally and $\xi$ is bounded on any interval $[t_k, t_{k+1}]$
  ($\xi$ is continuous).
  Moreover for all $l\geq1$, with $\tau\in[t_k, t_{k+1}]$, $(t_{k+1}-\tau)^l\xi(\tau-h_n)-\xi(t)) = \mathcal{O}(h_n)$.
  Thus $\int_{t_{k}}^{t_{k+1}}\! (t_{k+1}-\tau)^l(\xi(\tau-h_n)-\xi(t))\mathrm{d}\tau = \mathcal{O}(h_n^2)$.
  Then~\eqref{eq:cv_1} can be rewritten as:
  \begin{align}
   \bar{u}_k^s-u^s(t) &= -(I+\mathcal{O}(h_n))\left(\int_{t_{k}}^{t_{k+1}}\!
   h_n^{-1}(\xi(\tau-h_n)-\xi(t))\mathrm{d}\tau + \mathcal{O}(h_n)\right).
   \intertext{Taking the supremum norm yields:}
  %
   \|\bar{u}_k^s-u^s(t)\|_\infty&\leq\|I + \mathcal{O}(h_n))\|_\infty
   \left( \max_i\sup_{t\in[t_k,t_{k+1}]}|\xi_i(t^\prime_i-h)-\xi_i(t)| + \|\mathcal{O}(h_n)\|_\infty\right)\\
   &\leq \max_i\sup_{t\in[t_k,t_{k+1}]}|\xi_i(t^\prime_i-h_n)-\xi_i(t)| + \mathcal{O}(h_n).\label{eq:prop_cv_last}
  \end{align}
  Since $\xi$ is uniformly continuous, for every $\varepsilon>0$, there exists $\delta>0$ such that for all $t_1, t_2\in\R$,
  $|t_1 - t_2|\leq\delta$ implies $\|\xi(t_1)-\xi(t_2)\|\leq\varepsilon$. Then the right-hand side of \eqref{eq:prop_cv_last}
  is converging to $0$ as $h_n\to0$. Since this is true for all $t\in S$, the proof is complete.
\end{proof}
%
It is also interesting to study the convergence of the variation of the control variable,
which may be thought of as a measure of the control input chattering. Let us first define the variation
of a function in some special cases. The material in Definitions~\ref{de:VarVecStep} and \ref{de:VarVecC1}
is adapted from \cite{ambrosio2000functions}.
\begin{de}\label{de:VarVecStep}
 Let $f\colon\R\to\R^m$ be a right-continuous step function, discontinuous at finitely many time instants $t_k$
 and $t_0, T\in\R$, $t_0<T$. Then the variation of $f$ on $[t_0, T]$ is defined as:
 \begin{equation}
  \mathrm{Var}_{t_0}^T(f) \coloneqq \sum_{k}\|f(t_{k})-f(t_{k-1})\|,\label{eq:defVarStep}
 \end{equation}
 with $k\in\N$ such that $t_k\in(t_0,T]$.
\end{de}
\begin{de}\label{de:VarVecC1}
 Let $f\colon\R\to\R^m$ be a continuously differentiable function, with bounded derivatives and let $t_0, T\in\R$, $t_0<T$.
 Then the variation of $f$ on $[t_0, T]$ is defined as:
 \begin{equation}
  \mathrm{Var}_{t_0}^T(f) \coloneqq \int_{t_0}^{T}\!\|\dot{f}(\tau)\|\mathrm{d}\tau.
  \label{eq:defVarC1}
 \end{equation}
\end{de}
\begin{prop}
 Let $\{h_n\}_{n\in\N}$ be any strictly decreasing sequence of positive numbers converging to $0$ with $h_0<h^*$.
 Suppose that $CB$ is positive-definite, and $\xi$ is a real-valued continuously differentiable with bounded derivative function.
 Let $\alpha$ be chosen such that the conditions of Proposition~\ref{prop:pert} are verified for each $h_n$.
 Then $\displaystyle\lim_{h_n\to 0}\mathrm{Var}_{T^*}^T(\bar{u}^s) = \mathrm{Var}_{T^*}^T(u^s)$.
\end{prop}
\begin{proof}
 Let $\{t_k\}$ be a sequence such that for all $k\in\N$, $t_{k+1}-t_k=h_n$.
 Let us recall that, with the implicit controller defined in Equations~\eqref{eq:dt_sys_pert}
 and~\eqref{lcp:inexact}, the reaching phase duration is bounded and
 that the sliding phase is established at $t=T^*$ if $h$ is small enough.
 Let $\mathrm{Var}_{T^*}^T(\bar{u}^s)$ denote the variation of the function $\bar{u}^s$ on the time interval $[T^*,T]$.
 The variation of the step function $\bar{u}^s$ is:
 \begin{gather}
  \hspace{-0.5em}\mathrm{Var}_{t_0}^T(\bar{u}^s) = \sum_{k}\|\bar{u}^s_{k+1}-\bar{u}^s_{k}\|
  = \sum_{k}\|(CB^*)^{-1}C\int_{t_{k}}^{t_{k+1}}\!e^{A(t_{k+1}-\tau)}B\xi(\tau)\mathrm{d}\tau-(CB^*)^{-1}C\int_{t_{k-1}}^{t_{k}}\!e^{A(t_{k}-\tau)}B\xi(\tau)\mathrm{d}\tau\|\\
  = \sum_{k}\|(CB^{*})^{-1}\left(\int_{t_{k}}^{t_{k+1}}\!CB(\xi(\tau)-\xi(\tau-h_n))\mathrm{d}\tau + \int_{t_{k}}^{t_{k+1}}\!C(e^{A(t_{k+1}-\tau)}-I)B(\xi(\tau) - \xi(\tau-h_n)) \right)\|.
 \end{gather}
  Using the approximation for $(CB^{*})^{-1}$ obtained in \eqref{eq:dev_taylor_CB}, we get:
  \begin{align}
  \mathrm{Var}_{t_0}^T(\bar{u}^s) &\leq 
   (I+\mathcal{O}(h_n))\sum_{k}\|\int_{t_{k}}^{t_{k+1}}\!\frac{\xi(\tau)-\xi(\tau-h_n)}{h_n}\mathrm{d}\tau\| \\
   &+ (I+\mathcal{O}(h_n))\sum_{k}\| (CB)^{-1}\int_{t_{k}}^{t_{k+1}}\!C(e^{A(t_{k+1}-\tau)}-I)B\frac{\xi(\tau) - \xi(\tau-h_n)}{h_n}\mathrm{d}\tau\| \\
  &\leq 
   (I+\mathcal{O}(h_n))\int_{t_0}^{T}\!\|\frac{\xi(\tau)-\xi(\tau-h_n)}{h_n}\|\mathrm{d}\tau \\
   &+ (I+\mathcal{O}(h_n))\sum_{k}\| (CB)^{-1}\int_{t_{k}}^{t_{k+1}}\!C(e^{A(t_{k+1}-\tau)}-I)B\frac{\xi(\tau) - \xi(\tau-h_n)}{h_n}\mathrm{d}\tau\|.\label{eq:var_est_tmp}
 \end{align}
 Using standard estimation inequalities, we obtain the following estimation for all integrals in the sum:
 \begin{gather}
  \|\int_{t_{k}}^{t_{k+1}}\!C(e^{A(t_{k+1}-\tau)}-I)B\frac{\xi(\tau) - \xi(\tau-h_n)}{h_n}\mathrm{d}\tau\|\leq
  \int_{t_{k}}^{t_{k+1}}\!\|C\|\|B\|(e^{\|A\|h_n}-1)\|\frac{\xi(\tau) - \xi(\tau-h_n)}{h_n}\|\mathrm{d}\tau.
 \end{gather}
 Then we can group the terms in \eqref{eq:var_est_tmp} as:
 \begin{multline}
  \mathrm{Var}_{t_0}^T(\bar{u}^s) \leq (I+\mathcal{O}(h_n))\int_{t_0}^{T}\!\|\frac{\xi(\tau)-\xi(\tau-h_n)}{h_n}\|\mathrm{d}\tau \\
  + (I+\mathcal{O}(h_n))\|(CB)^{-1}\|\int_{t_0}^{T}\!\|C\|\|B\|(e^{\|A\|h_n}-1)\|\frac{\xi(\tau) - \xi(\tau-h_n)}{h_n}\|\mathrm{d}\tau.\label{eq:var_est}
 \end{multline}
 Since $\xi$ is continuously differentiable, $\int_{t_0}^{T}\!\|\frac{\xi(\tau)-\xi(\tau-h_n)}{h_n}\|\mathrm{d}\tau$ converges
 to $\int_{t_0}^{T}\!\|\dot{\xi}(\tau)\|\mathrm{d}\tau=\mathrm{Var}_{t_0}^T({u}^s)$ as $h_n\to 0$.
 All the other terms in the right-hand side of \eqref{eq:var_est} converge to $0$ as $h_n\to 0$.
 %
\end{proof}

\subsection{Comparison with saturated SMC}

\begin{figure}[htb]
  \ffigbox{}{\CommonHeightRow{
  \begin{subfloatrow}[2]
   \ffigbox[\FBwidth-4.5pt]{\includegraphics{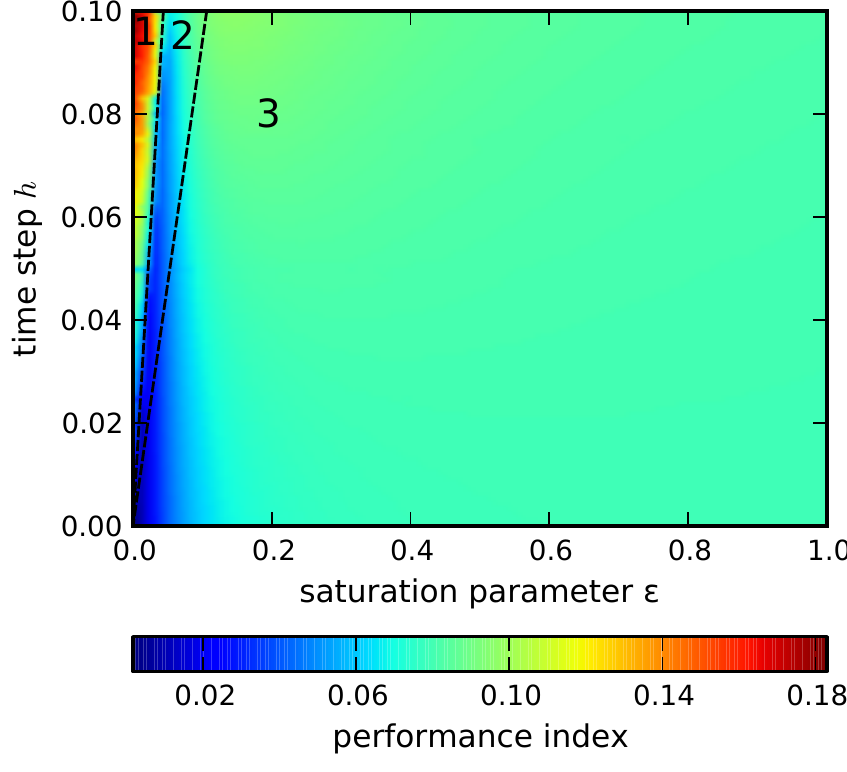}}
  {\caption{Simulation results with $100$ regularly spaced values for the timestep $h$ and $100$ logarithmically spaced
   values for the saturation parameter $\varepsilon$.}
   \label{fig:perf-sat-chattering}}
   \ffigbox[\FBwidth-4.5pt]{\includegraphics{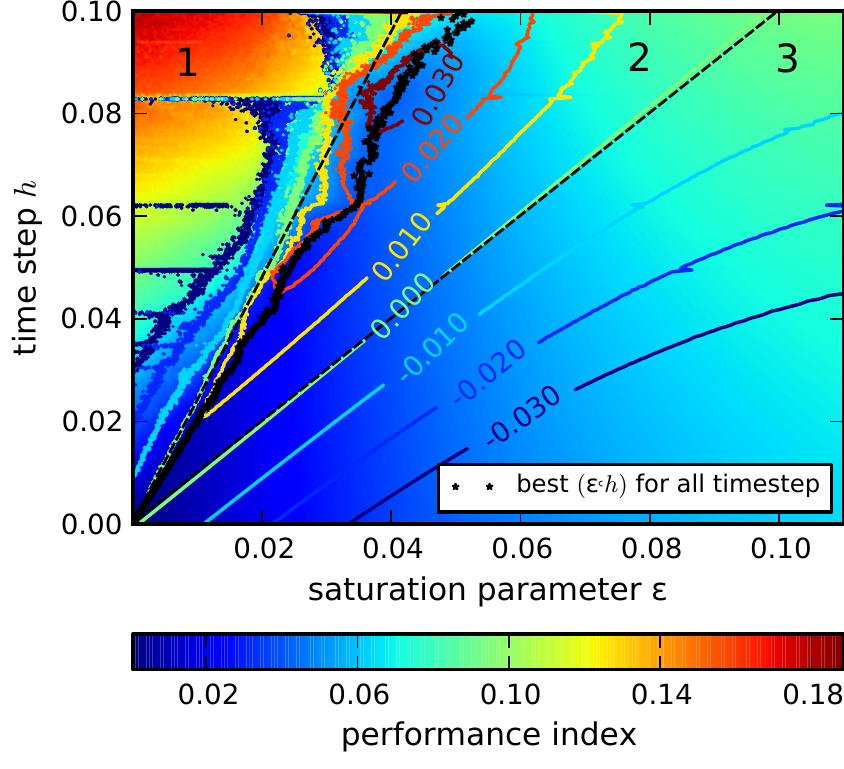}}
   {\caption{Detail of Fig.~\ref{fig:perf-sat-chattering}, $300$ values for $h$ and $1000$
   values for $\varepsilon$, forming a regular grid.
  Level sets were also added to show the difference in performance between the implicit discretization and
 the explicit one with saturation. If the difference is positive, the explicit saturated control is performing better
than the implicit one.}
\label{fig:pert-sat-chattering_z}}
\end{subfloatrow}}
 \caption{Simulation results of a perturbed system controlled using sliding mode with a saturation.
    The performance index is the sum of the $|\sigma_k|$ for the last $20\second$.}
    \label{fig:sat-chattering}}
\end{figure}
\begin{figure}[htb]
  \ffigbox{}{\CommonHeightRow{
  \begin{subfloatrow}[2]
   \ffigbox[\FBwidth-4.5pt]{\includegraphics{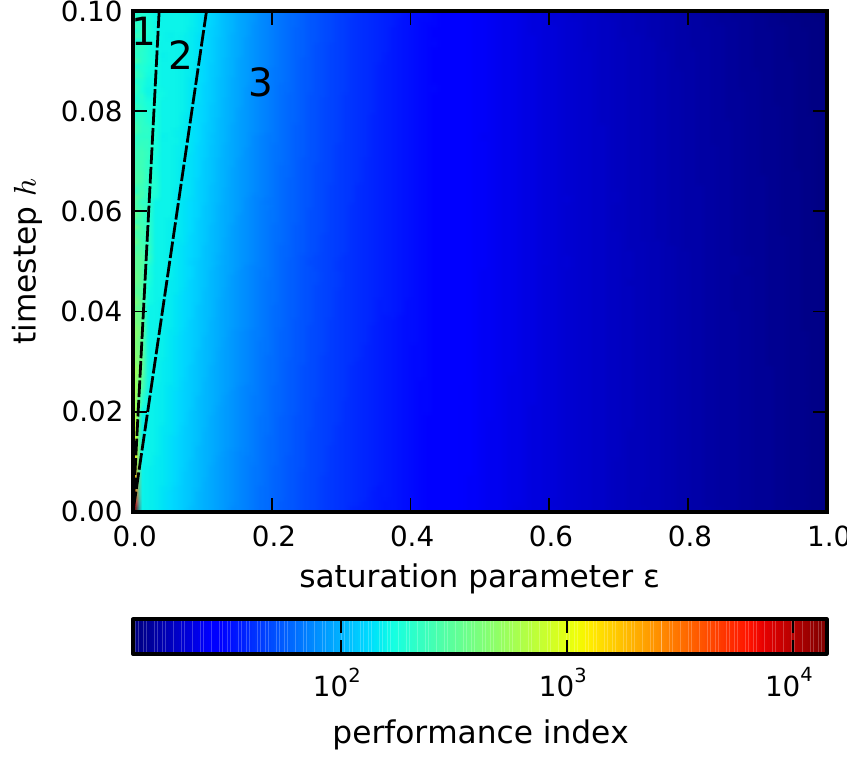}}
   {\caption{Simulation results with $100$ regularly spaced values for the timestep $h$ and $100$ logarithmically spaced
   values for the saturation parameter $\varepsilon$.}
   \label{fig:perf-sat-cost}}
   \ffigbox[\FBwidth-4.5pt]{\includegraphics{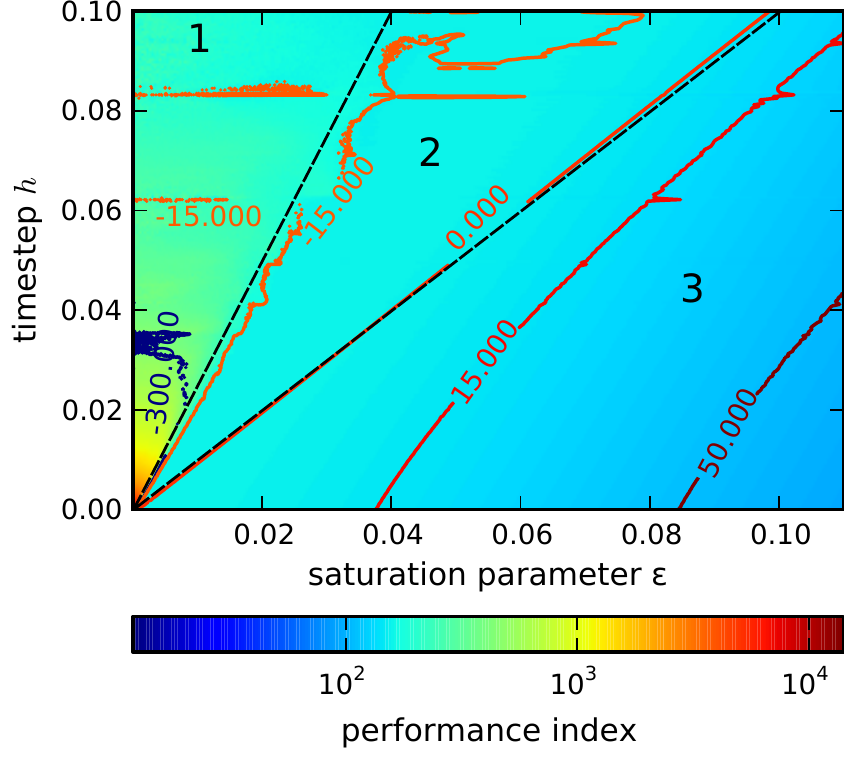}}
   {\caption{Detail of Fig.~\ref{fig:perf-sat-cost},
  $300$ values for $h$ and $1000$ values for $\varepsilon$, forming a regular grid.
  Level sets were also added to show the difference in performance between the implicit discretization and
 the explicit one with saturation. If the difference is positive, the explicit saturated control is performing better
than the implicit one.}
\label{fig:pert-sat-cost_z}}
\end{subfloatrow}}
 \caption{Simulation results of the same perturbed system controlled using sliding mode with a saturation.
  The performance index is the sum of the $|\bar{u}^s_{k+1}-\bar{u}^s_{k}|$ for the last $20\second$.}
  \label{fig:sat-cost}}
\end{figure}
We simulate the system~\eqref{sys:2Dunstable} with a perturbation $\xi(t) = \sin 4\pi t$. Instead of using a discontinuous
control, we use the following input: $u^s(t) = -\sate(\sigma(t))$ with $\sate(x) = \begin{cases}
 x/\varepsilon &\text{if } |x| \leq \varepsilon\\
 \sgn(x) &\text{if } |x| > \varepsilon
\end{cases}$. As in Section~\ref{sec:sim}, each simulation lasts $150\second$. We use two metrics to measure the
performance of the different controllers.
To measure the (output) chattering, due to the discretization and the perturbation, we
sum the absolute value of the sliding variable $\bar{\sigma}_k$ for
the last $20\second$: $C_1\coloneqq\sum_k|\bar{\sigma}_k|$.
To measure the control effort (or input chattering), we measure the
variation of the control for the last $20\second$: $C_2\coloneqq\mathrm{Var}_{T-20}^T(\bar{u}^s)$. Each quantity defines the
performance index in Fig.~\ref{fig:sat-chattering} or~\ref{fig:sat-cost}. We choose to consider only the last
$20\second$ of each simulation to capture the behaviour near the sliding manifold. Let us recall that without perturbation, the
implicit controller always supersedes the saturated explicit one, since it suppresses numerical chattering and $\bar{u}^s_k = 0$
in the discrete-time sliding phase. With both index, we can divide the space into 3 cones, numbered $1$, $2$ and $3$ in
Fig.~\ref{fig:sat-chattering} and~\ref{fig:sat-cost}.
This separation helps us to compare both controllers. In Fig.~\ref{fig:perf-sat-chattering}
the performance in term of chattering is presented. For large values of $\varepsilon$, the chattering does not change
when the timestep varies: the control action does not attenuate the effect of the perturbation.
With a small $\varepsilon$, the behaviour is richer, as depicted in Fig.\ref{fig:pert-sat-chattering_z}.
On Fig.~\ref{fig:pert-sat-chattering_z}, the overall best performance is obtained with small values for both
$\varepsilon$ and $h$. However for small values of $\varepsilon$, the performance can degrade rapidly if the
sampling period $h$ is not small enough, as seen in region $1$.
The dark points indicate for each value of $h$ the pair $(\varepsilon, h)$ of parameters yielding the best performance.
It seems that there is a linear relationship between those values. However it is unclear if this observation on
one particular system remains valid with a different perturbation.
The level sets in Fig.~\ref{fig:pert-sat-chattering_z} are used to compare the performance of
the implicit and the saturated explicit controllers. On Fig.~\ref{fig:sat-cost}, the performance in terms of control cost
is presented. The best performance is achieved for large $\varepsilon$ since the slope of the saturated function is gentle.
On the other hand in Fig.~\ref{fig:pert-sat-cost_z}, with a small $\varepsilon$, the cost increases
and explodes with $\varepsilon$ close to $0$, as in region $1$.
The level sets indicate the difference between the costs of the 2 different controllers.
It is worth noting that on region $2$ where the saturated controller is better in Fig.~\ref{fig:pert-sat-chattering_z},
it has a higher cost in term of control (Fig.~\ref{fig:pert-sat-cost_z}).
In region $3$, where the saturated controller performs less in terms of chattering (Fig.~\ref{fig:perf-sat-chattering}),
it has a smaller cost in terms of control (Fig.~\ref{fig:perf-sat-cost}).
Indeed with a large $\varepsilon$, the control input is small when the closed-loop system is close to the sliding manifold.
The cost is then very small, but the disturbance is not attenuated at all. The implicit controller appeals to us as the best
compromise between the input and output chattering. It is also very easy to use, since it requires no particular tuning
with respect to the timestep or the perturbation.

\section{Conclusion}\label{sec:conclusions}

In this article several time discretizations of the classical ECB-SMC method are analysed,
from the point of view of their ability to alleviate or suppress the numerical chattering,
and to guarantee the finite-time reachability of the sliding surface.
A new discrete-time sliding mode control scheme is also proposed.
The analysis is led from analytical estimations, as well as numerical simulations
obtained with the INRIA software package {\sc siconos}.
In particular the influence of the discretization method of the state-continuous equivalent controller is studied,
as well as the one of the discontinuous part of the input (explicit versus implicit discretizations).
The nominal and perturbed cases are considered. The simulation results indicate that the use of an explicit discretization
for the discontinuous part of the input yields numerical chattering. This is not the case when using an implicit
discretization. We also provide an example where the use of an explicit discretization of
$u^{eq}$ makes the closed-loop system diverge, whereas with the other methods it attains the sliding surface.
The issues related to the Lyapunov stability of the discrete-time sliding variable are also studied,
using the monotonicity properties of the underlying discontinuous (set-valued) controller.
Further works will include conducting experimental studies and also improvements in the perturbation attenuation.

\section*{Bibliography}
\bibliographystyle{IEEEtran}
\bibliography{ref}

\begin{thebibliography}{10}
\providecommand{\url}[1]{#1}
\csname url@samestyle\endcsname
\providecommand{\newblock}{\relax}
\providecommand{\bibinfo}[2]{#2}
\providecommand{\BIBentrySTDinterwordspacing}{\spaceskip=0pt\relax}
\providecommand{\BIBentryALTinterwordstretchfactor}{4}
\providecommand{\BIBentryALTinterwordspacing}{\spaceskip=\fontdimen2\font plus
\BIBentryALTinterwordstretchfactor\fontdimen3\font minus
  \fontdimen4\font\relax}
\providecommand{\BIBforeignlanguage}[2]{{%
\expandafter\ifx\csname l@#1\endcsname\relax
\typeout{** WARNING: IEEEtran.bst: No hyphenation pattern has been}%
\typeout{** loaded for the language `#1'. Using the pattern for}%
\typeout{** the default language instead.}%
\else
\language=\csname l@#1\endcsname
\fi
#2}}
\providecommand{\BIBdecl}{\relax}
\BIBdecl

\bibitem{sarpturk1987stability}
S.~Sarpturk, Y.~Istefanopulos, and O.~Kaynak, ``On the stability of
  discrete-time sliding mode control systems,'' \emph{Automatic Control, IEEE
  Transactions on}, vol.~32, no.~10, pp. 930--932, 1987.

\bibitem{drakunov1989discrete}
S.~Drakunov and V.~Utkin, ``On discrete-time sliding modes,'' in \emph{Proc. of
  {\textsc{IFAC}} Nonlinear Control System Design Conf.}, 1989, pp. 273--278.

\bibitem{furuta1990sliding}
K.~Furuta, ``Sliding mode control of a discrete system,'' \emph{Systems \&
  Control Letters}, vol.~14, no.~2, pp. 145--152, 1990.

\bibitem{utkin1994sliding}
V.~Utkin, ``Sliding mode control in discrete-time and difference systems,'' in
  \emph{Variable Structure and Lyapunov Control}, ser. Lecture Notes in Control
  and Information Sciences.\hskip 1em plus 0.5em minus 0.4em\relax Springer,
  1994, vol. 193, pp. 87--107.

\bibitem{gao1995discrete}
W.~Gao, Y.~Wang, and A.~Homaifa, ``Discrete-time variable structure control
  systems,'' \emph{Industrial Electronics, IEEE Transactions on}, vol.~42,
  no.~2, pp. 117--122, 1995.

\bibitem{golo2000robust}
G.~Golo and {\u C}.~Milosavljevi{\'c}, ``Robust discrete-time chattering free
  sliding mode control,'' \emph{Systems \& Control Letters}, vol.~41, no.~1,
  pp. 19--28, 2000.

\bibitem{milosavljevic1985general}
{\u C}.~Milosavljevi{\'c}, ``General conditions for the existence of a
  quasi-sliding mode on the switching hyperplane in discrete variable structure
  systems,'' \emph{Automation and Remote Control}, vol.~46, no.~3, pp.
  307--314, 1985.

\bibitem{galias2006complex}
Z.~Galias and X.~Yu, ``Complex discretization behaviors of a simple
  sliding-mode control system,'' \emph{Circuits and Systems II: Express Briefs,
  IEEE Transactions on}, vol.~53, no.~8, pp. 652--656, 2006.

\bibitem{galias2008analysis}
------, ``Analysis of zero-order holder discretization of two-dimensional
  sliding-mode control systems,'' \emph{Circuits and Systems II: Express
  Briefs, IEEE Transactions on}, vol.~55, no.~12, pp. 1269--1273, 2008.

\bibitem{wang2009zoh}
B.~Wang, X.~Yu, and G.~Chen, ``{\textsc{ZOH}} discretization effect on
  single-input sliding mode control systems with matched uncertainties,''
  \emph{Automatica}, vol.~45, no.~1, pp. 118--125, 2009.

\bibitem{acary2010implicit}
V.~Acary and B.~Brogliato, ``Implicit {Euler} numerical scheme and
  chattering-free implementation of sliding mode systems,'' \emph{Systems \&
  Control Letters}, vol.~59, no.~5, pp. 284--293, 2010.

\bibitem{acary2012chattering}
V.~Acary, B.~Brogliato, and Y.~Orlov, ``Chattering-free digital sliding-mode
  control with state observer and disturbance rejection,'' \emph{Automatic
  Control, IEEE Transactions on}, vol.~57, no.~5, pp. 1087--1101, 2012.

\bibitem{plestan2012advances}
F.~Plestan, V.~Bregeault, A.~Glumineau, Y.~Shtessel, and E.~Moulay, ``Advances
  in high order and adaptive sliding mode control--theory and applications,''
  in \emph{Sliding Modes after the First Decade of the 21st Century}, ser.
  Lecture Notes in Control and Information Sciences.\hskip 1em plus 0.5em minus
  0.4em\relax Springer, 2012, vol. 412, pp. 465--492.

\bibitem{defoort2009novel}
M.~Defoort, T.~Floquet, A.~Kokosy, and W.~Perruquetti, ``A novel higher order
  sliding mode control scheme,'' \emph{Systems \& Control Letters}, vol.~58,
  no.~2, pp. 102--108, 2009.

\bibitem{cottle2009linear}
R.~Cottle, J.-S. Pang, and R.~Stone, \emph{The Linear Complementarity Problem},
  ser. Classics in Applied Mathematics.\hskip 1em plus 0.5em minus 0.4em\relax
  Society for Industrial Mathematics, 2009, no.~60.

\bibitem{edwards1998sliding}
C.~Edwards and S.~Spurgeon, \emph{Sliding Mode Control: Theory and
  Applications}, ser. Systems and Control Book Series.\hskip 1em plus 0.5em
  minus 0.4em\relax CRC Press, 1998, vol.~7.

\bibitem{wang2008zoh}
B.~Wang, X.~Yu, and X.~Li, ``{\textsc{ZOH}} discretization effect on
  higher-order sliding-mode control systems,'' \emph{Industrial Electronics,
  IEEE Transactions on}, vol.~55, no.~11, pp. 4055--4064, 2008.

\bibitem{facchinei2003finite}
F.~Facchinei and J.-S. Pang, \emph{Finite-Dimensional Variational Inequalities
  and Complementarity Problems}, ser. Springer Series in Operations
  Research.\hskip 1em plus 0.5em minus 0.4em\relax Springer, 2003.

\bibitem{acary2008numerical}
V.~Acary and B.~Brogliato, \emph{Numerical Methods for Nonsmooth Dynamical
  Systems: Applications in Mechanics and Electronics}, ser. Lecture Notes in
  Applied and Computational Mechanics.\hskip 1em plus 0.5em minus 0.4em\relax
  Springer Berlin Heidelberg, 2008, vol.~35.

\bibitem{furuta2002discrete}
K.~Furuta and Y.~Pan, ``Discrete-time variable structure control,'' in
  \emph{Variable Structure Systems: Towards the 21st Century}, ser. Lecture
  Notes in Control and Information Sciences, J.-X. Yu~X., Xu, Ed.\hskip 1em
  plus 0.5em minus 0.4em\relax Springer Berlin Heidelberg, 2002, vol. 472, pp.
  57--81.

\bibitem{lin2004total}
C.-F. Lin and W.-C. Su, ``A total chattering-free sliding mode control for
  sampled-data systems,'' in \emph{American Control Conference, 2004.
  Proceedings of the 2004}, vol.~3.\hskip 1em plus 0.5em minus 0.4em\relax
  IEEE, 2004, pp. 1940--1945.

\bibitem{acary2007introduction}
\BIBentryALTinterwordspacing
V.~Acary and F.~P{\'e}rignon, ``\BIBforeignlanguage{Anglais}{{An introduction
  to \textsc{Siconos}}},'' INRIA, Rapport Technique RT-0340, 2007. [Online].
  Available: \url{http://hal.inria.fr/inria-00162911}
\BIBentrySTDinterwordspacing

\bibitem{hunter2007matplotlib}
J.~D. Hunter, ``Matplotlib: a {2D} graphics environment,'' \emph{Computing in
  Science \& Engineering}, pp. 90--95, 2007.

\bibitem{hinrichsen2005mathematical}
D.~Hinrichsen and A.~J. Pritchard, \emph{Mathematical Systems Theory I}, ser.
  Texts in Applied Mathematics.\hskip 1em plus 0.5em minus 0.4em\relax Springer
  Berlin Heidelberg, 2005, vol.~48.

\bibitem{utkin1992sliding}
V.~Utkin, \emph{Sliding Modes in Control and Optimization}, ser. Communications
  and Control Engineering.\hskip 1em plus 0.5em minus 0.4em\relax Springer
  Berlin, 1992.

\bibitem{bernstein2005matrix}
D.~Bernstein, \emph{Matrix Mathematics}.\hskip 1em plus 0.5em minus 0.4em\relax
  Princeton University Press Princeton, NJ, 2005.

\bibitem{horn1991topics}
R.~A. Horn and C.~R. Johnson, \emph{Topics in Matrix Analysis}.\hskip 1em plus
  0.5em minus 0.4em\relax Cambridge University Press, 1994.

\bibitem{ambrosio2000functions}
L.~Ambrosio, N.~Fusco, and D.~Pallara, \emph{Functions of Bounded Variation and
  Free Discontinuity Problems}.\hskip 1em plus 0.5em minus 0.4em\relax
  Clarendon Press Oxford, 2000.

\end{thebibliography}
\end{document}